\documentclass[dvips,preprint]{imsart}
\usepackage{amsmath}
\usepackage{amsfonts}
\usepackage{mathrsfs}
\usepackage{color}
\usepackage{ amsmath, amsfonts, amssymb, amsthm, amscd}
\usepackage[T1]{fontenc}
\usepackage[english]{babel}

\usepackage[latin1]{inputenc}

\fontsize{11pt}{14.5pt}\selectfont

\def \bbR{\mathbb{R}}
\def \R{\mathbb{R}}
\def \P{\mathbb{P}}

\def \E{I\!\!E}

\def \bf{\textbf}
\def \it{\textit}

\def \cB{\mathcal{B}}

\newcommand{\dtmur}{\widetilde\mu(dr,de)}
\DeclareMathSymbol{\leqslant}{\mathalpha}{AMSa}{"36} 
\DeclareMathSymbol{\geqslant}{\mathalpha}{AMSa}{"3E} 
\renewcommand{\leq}{\;\leqslant\;}                   
\renewcommand{\geq}{\;\geqslant\;}                   

\newcommand{\integ}[2]{\displaystyle \int_{#1}^{#2}}

\newcommand{\beq}{\begin{eqnarray}}
\newcommand{\eeq}{\end{eqnarray}}
\newcommand{\beqn}{\begin{eqnarray*}}
\newcommand{\eeqn}{\end{eqnarray*}}

\newtheorem{theorem}{Theorem}
\newtheorem{lemma}{Lemma}
\newtheorem{proposition}{Proposition}

\newtheorem{definition}{Definition}
\newtheorem{corollary}{Corollary}
\newtheorem{remark}{Remark}

\begin{document}
\begin{frontmatter}
\title{The  obstacle problem for semilinear parabolic partial integro-differential
equations}
\date{}
\runtitle{}

\author{\fnms{Anis}
 \snm{MATOUSSI}\corref{}\ead[label=e1]{anis.matoussi@univ-lemans.fr}}
 \thankstext{t3}{Research partly supported by the Chair {\it Financial Risks} of the {\it Risk Foundation} sponsored by Soci\'et\'e G\'en\'erale, the Chair {\it Derivatives of the Future} sponsored by the {F\'ed\'eration Bancaire Fran\c{c}aise}, and the Chair {\it Finance and Sustainable Development} sponsored by EDF and Calyon }
\address{
 Universit\'e du Maine \\
 Institut du Risque et de l'Assurance\\
Laboratoire Manceau de Math\'ematiques\\\printead{e1}
 }

\author{\fnms{Wissal}
 \snm{SABBAGH}\corref{}\ead[label=e2]{wissal.sabbagh.etu@univ-lemans.fr}}
\address{Universit\'e du Maine\\
 Institut du Risque et de l'Assurance\\
 Laboratoire Manceau de Math\'ematiques\\
\printead{e2}
}

\author{\fnms{Chao}
 \snm{ZHOU}\corref{}\ead[label=e3]{matzc@nus.edu.sg}}
\address{National University of Singapore \\
Department of Mathematics\\
\printead{e3}
}


\runauthor{A. Matoussi, W. Sabbagh, C. Zhou}

\vspace{3mm}

%
%

\begin{abstract} 
 This paper presents a probabilistic interpretation for the  weak Sobolev solution of the obstacle problem for semilinear parabolic partial integro-differential equations (PIDEs). 
 The results of L\'eandre \cite{l} concerning the homeomorphic property for the solution of SDEs with jumps are used to construct random test functions for the variational equation for such PIDEs. This results in the natural connection with the associated Reflected Backward Stochastic Differential Equations with jumps (RBSDEs), namely Feynman Kac's formula for the solution of the PIDEs. Moreover it gives an application to the pricing and hedging of contingent claims with constraints in the wealth or portfolio processes in financial markets including jumps.\\

\end{abstract}

\end{frontmatter}

\section{Introduction}
Our main interest is to study the 
following partial integro-differential equations (in short PIDEs)
of parabolic type:
\begin{equation}\label{pde1}
\begin{array}{ll}
(\partial_t+\mathcal L)u(t,x)+f(t,x,u(t,x),\sigma^\ast\nabla
u(t,x),u(t,x+\beta(x,\cdot))-u(t,x))=0\\
\end{array}
 \end{equation}
 over the time interval $[0,T]$, with a given final condition $u_T = g$, $f$ is a nonlinear function and  $\mathcal L={\mathcal K}_1+{\mathcal K}_2$ is the second order integro-differential operator associated with a jump diffusion which is defined
 component by component with
 \begin{equation}\label{operator}
 \begin{array}{lll}
 {\mathcal K}_1\varphi(x)&=&\displaystyle\sum_{i=1}^{d}b^i(x)\frac{\partial}{\partial
 x_i}\varphi(x)+\frac{1}{2}\sum_{i,j=1}^{d}a^{ij}(x)\frac{\partial^2}{\partial
 x_i\partial x_j}\varphi(x)\mbox{ and }\\
{\mathcal
K}_2\varphi(x)&=&\displaystyle\int_{\E}\Big(\varphi(x+\beta(x,e))-\varphi(x)-\sum_{i=1}^{d}\beta^i(x,e)\frac{\partial}{\partial
 x_i}\varphi(x)\Big)\lambda(de),\quad \varphi \in C^2(\bbR^d).
\end{array}
\end{equation}

This class of PIDEs appears in the pricing and hedging contigent claims in financial markets including jumps. Matache, von
Petersdorff and Schwab \cite{MPS04} have studied a particular case where $f$ is linear  in $(y,z)$ and not depends on $v$ (the jump size variable). They have shown the existence and
uniqueness of the Sobolev solution of the variational form of some
types of PIDEs, stemming from pricing problems in L\'evy markets. They
used an analytic method in order to derive a numerical schema based
on the wavelet Galerkin method.

Our nonlinear PIDEs (\ref{pde1})  include the case of pricing of contingent claims with constraints in the wealth or portfolio processes. As an example, hedging claims with higher interest rate for borrowing may be considered in a financial market with jumps. El Karoui, Peng and Quenez \cite{ElPQ97} have studied this example  in a continuous financial market where the non linear source function $f$ is given by $f(t,x,y,\tilde{z})=r_ty+\theta_t\sigma_t\tilde{z}-(R_t-r_t)(y-\sum_{i=1}^{n} \tilde{z}^i).$ \\
In the classical literature, the obstacle problem is related to the variational inequalities which were first studied by Mignot-Puel \cite{MignotPuel75}, and then by Michel Pierre \cite{Pierre, PIERRE} (see also Bensoussan-Lions \cite{BensoussanLions78} ,  Kinderleherer-Stampacchia \cite{KS80} and Bally et al  \cite{BCEF}). More recently,    
the pricing and hedging of American options in the Markovian case and the related  obstacle problem for PDEs, was studied by El Karoui et al \cite{Elk2} (see also El Karoui, Pardoux and Quenez \cite{ElkPaQ97}  and El Karoui, Hamad\`ene, Matoussi \cite{ElkMH08}) .\\
In the case where $f$  does not depend on $u$ and  $\nabla u$, , the equation \eqref{pde1} becomes a linear parabolic PDE. 
If $ h : [0,T] \times \mathbb{R}^d  \longrightarrow \mathbb{R} $ is a given function such that $ h (T,x) \leq g (x)$, we may roughly say that the solution of the obstacle problem for \eqref{pde1} is a function $ u \in \mathbf{L}^2 \big([0,T]; H^1 (\mathbb{R}^d) \big)$ such that
the following conditions are satisfied in $(0,T) \times \mathbb{R}^d$ :
{\small
\begin{equation}
\begin{split}
\label{OLPIDE}
& (i) \; \;    u \geq h , \quad dt\otimes dx - \mbox{a.e.},  \\
& (ii)\;\;  \partial_t u  +
 \;  \mathcal L  u  + f  \leq 0 \\
& (iii)\;\;   \big(u - h \big) \big( \partial_t u  +
 \; \mathcal L u  + f \big) = 0 .\\
 & (iv)  \; \;  u_T = g, \quad dx-\mbox{a.e.}
\end{split}
\end{equation}}
The relation $(ii)$ means that the distribution appearing in the LHS of the inequality is a non-positive measure. The relation
$(iii)$ is not rigourously stated. We may roughly say that one has $\partial_t u  +
 \; \mathcal L  u  + f   = 0$ on the set $ \{ u > h \}$. 
 
   In the case of the obstacle problem for PDEs (when the non local term operator $\mathcal K_2$ =0),  if one expresses the obstacle problem in terms of variational inequalities it should also be required that the solution has a minimality property (see Mignot-Puel \cite{MignotPuel75} or Bensoussan-Lions \cite{BensoussanLions78} p.250). The work of El Karoui et al \cite{Elk2} treats the obstacle problem for \eqref{pde1} within the framework of backward stochastic differential equations (BSDEs in short). Namely the equation \eqref{pde1} is considered with $f$ depending on $u$ and $\nabla u$ ,  $ \lambda = 0$ and $\beta =0$  and the obstacle $h$ is continuous. The solution is represented stochastically as a process  and the main new  object of this BSDE framework is a continuous increasing process that controls the set $\{u=h\}$. This increasing process determines in fact the measure from the relation $(ii)$.  Bally et al \cite{BCEF} (see also Matoussi and Xu \cite{MX08})  point out that the continuity of this process allows the classical notion of strong variational solution to be extended (see Theorem 2.2 of \cite{BensoussanLions78} p.238) and express the solution to the obstacle as a pair $(u, \nu)$ where $ \nu$ equals the LHS of $(ii)$ and is supported by the set  $\{u=h\}$. \\
Barles, Buckdahn and Pardoux \cite{bbp97} have provided  a  probabilistic interpretation for  the viscosity solution of (\ref{pde1})  by using 
a forward BSDE with jumps.  Situ \cite{situ} has studied  the Sobolev solution of
(\ref{pde1}) via an appropriate BSDE with jumps, whose method is
mainly based on Sobolev's embedding theorem.\\
More recently, Matoussi and Stoica \cite{MS10}  studied  the obstacle problem for parabolic quasilinear SPDEs and gave a probabilistic interpretation of  the reflected measure $\nu$  in terms of the associated increasing process which is a component of solution of reflected BSDEs. Such measures are called regular measures or Revuz measures. Their method is based on probabilistic  quasi-sure analysis. 

 Michel Pierre \cite{Pierre, PIERRE}  has studied parabolic PDEs with obstacles using the parabolic potential as a tool. He proved that the solution uniquely exists and is quasi-continuous with respect to socalled analytical capacity.  Moreover he gave a representation of the reflected measure $\nu$ in terms of the associated regular potential and the approach used  is based  on analytical quasi-sure analysis.  More recently,  Denis, Matoussi and  Zhang \cite{DMZ12} have extended the approach of  Michel Pierre \cite{Pierre, PIERRE}  
for the obstacle problem of quasilinear SPDEs.

 More precisely, our main interest is to consider the final condition to be a fixed function $g\in \mathbf{L}^2 \big(\mathbb{R}^d\big)$ and the obstacle $h$ be a  continuous function  $ h : [0,T] \times\mathbb{ R}^d \longrightarrow \mathbb{R }$. Then the obstacle problem for the equation \eqref{pde1} is defined as a  pair $ (u, \nu)$, where $ \nu$ is a regular measure concentrated on $\{u=h\}$ and $ u \in \mathbf{L}^2 \big( [0,T] \times R^d; R)\big)$ satisfies the following relations:
\small
\begin{equation}
\begin{split}
\label{OPDIE}
& (i') \; \;    u \geq h , \quad d\mathbb{P}\otimes dt\otimes dx - \mbox{a.e.},  \\
& (ii')\;\; \partial_t u (t,x) + 
 \;  \mathcal L  u (t,x)  + f(t,x,u (t,x),   \sigma^\ast\nabla
u(t,x),u(t,x+\beta(x,\cdot))-u(t,x)) = - \nu (dt,dx),  \\
& (iii')\;\;  \nu \big(u > h \big) =0, \quad a.s.,\\
 & (iv')  \; \;  u_T = g, \quad  dx-\mbox{a.e.}.
\end{split}
\end{equation}
$\nu$ represents the quantity which makes it possible to pass from inequality $(ii)$ to equality $(ii')$ and to get the uniqueness result for the obstacle problem.
In Section \ref{section:obstacle}, the rigorous  sense of the relation $(iii')$  which is based on the  probabilistic representation of the measure $\nu$  and  plays the role of  quasi-continuity of $u$ in this context will be explained. The main result of the paper is Theorem \ref{mr2} which ensures the existence and uniqueness of the solution $(u, \nu)$ of the obstacle problem for \eqref{pde1} using a probabilistic method based on reflected BSDEs with jumps. The proof is based on the penalization procedure.  It can be noted that the quasi-sure approaches for the PIDEs  (probabilistic   \cite{MS10} or analytical one  \cite{DMZ12}) are unsuccessful. It remains unclear, until now, how to define the analytical potential associated to the operator $\mathcal L$ specially for the non local operator $\mathcal K_2$. Therefore, it is not obvious how to define the associated analytical capacity. Thus, the stochastic flow method developed by Bally and Matoussi in \cite{bm} for a class of parabolic semilinear SPDEs is used in this context. \\ 
As a preliminary work,  first we present the existence  and uniqueness of Sobolev's solution of  PIDE \eqref{pde1} (without obstacle) and provide a probabilistic interpretation  by using  solution of BSDEs driven by a Brownian motion and an independent random measure.  
The concern is to solve our problem by developing a stochastic flow method based on the results of L\'eandre \cite{l} about the homeomorphic property for the solution of SDEs with jumps. The key element in \cite{bm} is to use  
the inversion of stochastic flow which transforms the variational
formulation of the PDEs to the associated BSDEs. Thus it plays the
same role as It\^o's formula in the case of the classical solution of
PDEs. 
Note that more recently, in \cite{jourdain07}  based on stochastic
flow arguments, the author shows that the probabilistic equivalent
formulation of Dupire's PDE is the Put-Call duality equality in
local volatility models including exponential L\'evy jumps. Also in
\cite{Mrad09}  and \cite{ElKM13}, the inversion of stochastic flow technics  are  used  for building
  a family of forward utilities  for a given optimal portfolio.

This paper is organized as following: in section 2, first the basic assumptions and the definitions of the solutions
for PIDEs are presented. We provide useful results on stochastic flow associated with the forward SDEs with jumps, then in this
 setting a class of  random test functions and their semimartingale decomposition are introduced. Finally, an equivalence norm result is given in the jump diffusion case.
   In section 3, we prove the existence and uniqueness results
   of the solution of our PIDEs  and give the associated probabilistic interpretation via the FBSDEs
   with jumps. The uniqueness is a consequence of the variational formulation of the PIDEs  written with random test functions and the uniqueness of the solution of the FBSDE.
    The existence of the solution is established by an approximation penalization  procedure, a priori estimates 
    and the equivalence norm results. In section 4,  we prove existence and uniqueness of  the solution of the obstacle problem for the PIDEs.  The proof of this result differs from that of Bally et al  \cite{BCEF} since we have to consider the stochastic flow associated with the forward jump diffusion process. In particular, the jump part appearing of the tightness result for the approximation measure has to be taken into account. 
    In the Appendix, we first give the proof of the equivalence norm results, then prove a regularity result for the BSDEs solution with respect to the
     time-state variable $(t,x)$, in order to relate the solution of BSDEs to the classical solution of our PIDEs. Finally, we give a proof of a  technical lemma which is crucial for the existence of the  regular measure part of the solution of our obstacle problem for PIDEs.
\section{Hypotheses and preliminaries}
\label{hypotheses}
Let  $T>0$  be a finite time horizon and $(\Omega, \mathcal{F},
\mathbb{P})$ be a complete probability space on which is defined two
independent processes:

- a $d$-dimensional Brownian motion $W_t = (W_t^1,\cdots, W_t^d)$;

- a Poisson random measure $\mu=\mu(dt,de)$ on
$([0,T]\times \E,{\mathcal B}([0,T])\otimes{\mathcal B}_{\E})$, where $\E=\R^{l}\setminus\{0\}$ is equipped with its Borel field ${\mathcal B}_{\E}$, with compensator $\upsilon(dt,de)=\lambda(de)dt$, such that $\{\tilde{\mu}([0,T]\times A)= (\mu-\upsilon)([0,T]\times A)\}_{t\geq 0}$ is a martingale for all $A\in {\mathcal B}_{\E}$ satisfying $\lambda(A)< \infty$. $\lambda$ is assumed to be a $\sigma-$ finite measure on 
$(\E,{\mathcal B}_{\E})$ satisfying 
$$ \int_{\E} (1\wedge |e|^2) \lambda(de) < +\infty$$
Denote $\widetilde{\mathcal P}={\mathcal P}\otimes{\mathcal B}_{\E}$ where $\mathcal P$ is the predictable $\sigma$-field on $\Omega\times [0,T]$.


Let $({\mathcal F}_t)_{t\geq 0}$ be the filtration generated by the
above two processes and augmented by the $P$-null sets of $\mathcal
F$. Besides let us define:
\medskip

- $|X|$ the Euclidean norm of a vector $X$;\smallskip

- $\mathbf{L}^2(\E,{\mathcal B}_{\E},\lambda; \R^n)$ (noted as
$\mathbf{L}^2_{\lambda}$ for convenience) the set of measurable functions
from $ (\E,{\mathcal B}_{\E},\lambda)$ to $\R^n$ endowed with the
topology of convergence in measure and for $v\in \mathbf{L}^2(\E,{\mathcal
B}_{\E},\lambda; \R^n)$ 
$$\begin{array}{cc}

\|v\|^2=\integ{\E}{}|v(e)|^2\lambda(de)\in \R^+\cup\{+\infty\};
\end{array}$$

-$\mathbf{L}^p_n({\mathcal F}_T)$ the space of $n$-dimensional
${\mathcal F}_T$-measurable random variables $\xi$ such that
$$\begin{array}{ll}
          \|\xi\|_{L^p}^p:=E(|\xi|^p)<+\infty;
  \end{array}$$

-${\mathcal H}^p_{n\times d}([0,T])$ the space of $\R^{n\times
d}$-valued ${\mathcal P}$-measurable process $Z=(Z_t)_{t\leq T}$
such that
$$\begin{array}{ll}
          \|Z\|_{{\mathcal H}^p}^p:=E[(\integ{0}{T}|Z_t|^2dt)^{p/2}]<+\infty;
  \end{array}$$

- ${\mathcal S}^p_n([0,T])$ the space of $n$-dimensional ${\mathcal
F}_t$-adapted c\`adl\`ag processes $Y=(Y_t)_{t\leq T}$ such that
$$\|Y\|_{{\mathcal S}^p}^p:=E[\sup_{t\leq T}|Y_t|^p]<+\infty;$$

- $\mathcal {A}^{p}_n(t,T)$ the space of $n$-dimensional ${\mathcal
F}_t$-adapted non-decreasing  c\`adl\`ag processes  $K=(K_t)_{t\leq T}$
such that
$$\|K\|_{{\mathcal A}^p}^p:=E[|K_T|^p]<+\infty;$$

-${\mathcal L}^p_n([0,T])$ the space of $\R^{n}$-valued $\widetilde
{\mathcal P}$-measurable mappings $V(\omega,t,e)$ such that
$$\begin{array}{ll}
          \|V\|_{{\mathcal L}^p}^p:=E[(\integ{0}{T}\|V_t\|^2dt)^{p/2}]
          =E[(\integ{0}{T}\int_{\E}|V_t(e)|^2\lambda(de)dt)^{p/2}]<+\infty.
  \end{array}$$

- $C^k_{l,b}(\R^p,\R^q)$ the set of $C^k$-functions which grow at most linearly at infinity and whose partial
derivatives of order less than or equal to $k$ are bounded.

- $\mathbf{L}_{\rho}^2\left( \mathbb{{R}}^d\right) $ will be the  basic Hilbert space  of our framework. We employ the usual notation for its scalar product and its norm,%
$$ \left( u,v\right)_{\rho} =\int_{\mathbb{R}^d}u\left( x\right) v\left(
x\right) \rho (x) dx,\;\left\| u\right\| _2=\left(
\int_{\mathbb{R}^d}u^2\left( x\right) \rho (x) dx\right) ^{\frac
12}. $$ where $ \rho$ is a continuous positive and integrable
weight function. We assume additionally that $\frac{1}{\rho}$ is locally integrable.
\\In general, we shall use the
notation
\[ (u,v)=\int_{\mathbb{R}^d} u(x)v(x)\, dx,\]
where $u$, $v$ are measurable functions defined in $\mathbb{R }^d$
and $uv \in \mathbf{L}^1 (\mathbb{R}^d )$.

Our evolution problem will be considered over a fixed time interval
$[0,T]$ and the norm for an element of $\mathbf{L}_{\rho}^2\left(
[0,T] \times \mathbb{{R}}^d\right) $ will be denoted by
$$\left\| u\right\| _{2,2}=\left(\int_0^T  \int_{\mathbb{R}^d} |u (t,x)|^2 \rho(x)dx dt \right)^{\frac 12}. $$
We usually omit the subscript when $n=1$. We assume the following hypotheses : \\[0.2cm]
\textbf{(A1)}  $g$ belongs to $\mathbf{L}_{\rho}^2(\R^d)$;
\\[0.2cm]
\textbf{(A2)} $f:[0,T]\times \R^d\times \R^m\times \R^{m\times
d}\times \mathbf{L}^2(\E,{\mathcal B}_{\E},\lambda; \R^m)\rightarrow\R^m$ is measurable in $(t,x,y,z,v)$ and
satisfies $ f^0 \in \mathbf{L}_{\rho}^2\left( [0,T] \times
\mathbb{{R}}^d\right) $ where $f^0 := f ( \cdot
,\cdot,0, 0,0)$.\\[0.2cm]
\textbf{(A3)} $f$ satisfies Lipschitz condition in $(y,z,v)$,
i.e., there exists a constant $C$ such that for any
$(t,x)\in[0,T]\times \R^d$ and $(y,z,v),(y',z',v')\in
\R^m\times\R^{m\times d}\times \mathbf{L}^2(\E,{\mathcal
B}_{\E},\lambda; \R^m)$:
$$|f(t,x,y,z,v)-f(t,x,y',z',v')|\leq C(|y-y'|+|z-z'|+\|v-v'\|);$$
\textbf{(A4)}  $b\in C^3_{l,b}(\R^d;\R^d)$, $\sigma\in
C^3_{l,b}(\R^d;\R^{d\times d})$, $\beta\,:\,\R^d\times \E\to \R^d$
be measurable and for all $e\in \E$, $\beta(\cdot,e)\in
C^3_{l,b}(\R^d;\R^d)$, and for some $K>0$ and for all $x\in \R^d$,
$e\in \E$,
$$|\beta(x,e)|\leq K(1\wedge|e|),\quad |D^{\alpha}\beta(x,e)|\leq
K(1\wedge|e|)\,\,\mathrm{for}\,\, 1\leq |\alpha|\leq 3,
$$
where $\alpha=(\alpha_1,\alpha_2,\cdots,\alpha_d)$ is a multi-index
and $|\alpha|=\alpha_1+\alpha_2+\cdots+\alpha_d.$ $D^{\alpha}$ is
the differential operator
$D^{\alpha}=\displaystyle\frac{\partial^{|\alpha|}}{(\partial^{\alpha_1}x_1)(\partial^{\alpha_2}x_2)\cdots(\partial^{\alpha_d}x_d)}$.


\subsection{Weak formulation for the partial differential-integral equations}
\label{definition:solution}

The space of test functions which we employ in the definition of
weak solutions of the evolution equations  \eqref{pde1} is $
\mathcal{D}_T  = \mathcal{C}^{\infty} (\left[0,T]\right) \otimes
\mathcal{C}_c^{\infty} \left(\mathbb{R}^d\right)$, where
$\mathcal{C}^{\infty} \left([0,T]\right)$ denotes the space of real
functions which can be extended as infinite differentiable functions
in the neighborhood of $[0,T]$ and $
\mathcal{C}_c^{\infty}\left(\mathbb{R}^d\right)$ is the space of
infinite differentiable functions with compact support in
$\mathbb{R}^d$. We denote the space of solutions  by
$$ {\mathcal H}_T:=\{u\in \mathbf{L}_{\rho}^2([0,T]\times \R^d) \; \big|\; \sigma^\ast\nabla u\in \mathbf{L}_{\rho}^2([0,T]\times \R^d)\} $$
endowed with the norm
$$\begin{array}{ll}
\|u\|_{{\mathcal H}_T}=
\Big(\displaystyle\int_{\mathbb{R}^d}\int_0^T[|u(s,x)|^2+|(\sigma^\ast\nabla
u)(s,x)|^2]ds\rho(x)dx \Big)^{1/2},
\end{array}
$$
where we denote the gradient by $\nabla u (t,x) = \big(\partial_1 u
(t,x), \cdot \cdot \cdot, \partial_d u (t,x) \big)$.

\begin{definition} We say that $ u \in \mathcal{H}_T $ is a Sobolev solution of PIDE  $\left(
\ref{pde1}\right) $ if the following
relation holds, for each $\phi \in \mathcal{D}_T ,$
\begin{equation}\label{wspde1}
\begin{array}{ll}
\displaystyle\int_t^T(u(s,x),\partial_s\phi(s,x))ds+(u(t,x),\phi(t,x))-(g(x),\phi(T,x))
-\int_t^T( u(s,x),\mathcal L^\ast \phi(s,x))ds
\\=\displaystyle\int_t^T(f(s,x,u(s,x),\sigma^\ast\nabla u(s,x),u(s,x+\beta(x,\cdot))-u(s,x)),\phi(s,x))ds.
\end{array}
\end{equation}
where ${\mathcal L}^\ast$ is the adjoint operator of ${\mathcal L}$.
We denote by $ u:=\mathcal{U }(g, f)$ such a solution.
\end{definition}
\subsection{Stochastic flow of diffeomorphisms and random test functions}
\label{Flow} In this section, we shall study the stochastic flow associated with the forward jump diffusion component. The main motivation is to generalize in the jump setting the flow technics which was first introduced in
\cite{bm} for the study of semilinear PDE's. Let
$(X_{t,s}(x))_{t\leq s\leq T}$ be the strong solution  of the equation:
 \begin{equation}\label{sde}
  X_{t,s}(x)=x+\integ{t}{s}b(X_{t,r}(x))dr+\integ{t}{s}\sigma(X_{t,r}(x))dW_r+\integ{t}{s}\int_{\E}
  \beta(X_{t,r-}(x),e)\dtmur.
\end{equation}
The existence and uniqueness of this solution was proved in
Fujiwara and Kunita \cite{fk}. Moreover, we have the following
properties  (see Theorem 2.2 and Theorem 2.3 in \cite{fk}):
\begin{proposition}\label{estimatesde}
For each $t>0$, there exists a version of $\{X_{t,s}(x);\,x\in
\R^d,\,s\geq t\}$ such that $X_{t,s}(\cdot)$ is a $C^2(\R^d)$-valued
c\`adl\`ag process. Moreover:

(i) $X_{t,s}(\cdot)$ and  $X_{0,s-t}(\cdot)$ have the same
distribution, $0\leq t\leq s$;

(ii) $X_{t_0,t_1},\,X_{t_1,t_2},\ldots,X_{t_{n-1}, t_n}$ are
  independent, for all $n\in \mathbb{N}$, $0\leq t_0<t_1<\cdots<t_n$;

(iii) $X_{t,r}(x)=X_{s,r} \circ X_{t,s} (x)$, $0\leq t<s<r$.

 Furthermore,
for all $p\geq 2$, there exists $M_p$ such that for all $0\leq t<s$,
$x,x'\in\R^d$, $h,h'\in\R\backslash{\{0\}}$,
$$\begin{array}{ll}
E(\sup\limits_{t\leq r\leq s}|X_{t,r}(x)-x|^p)\leq
M_p(s-t)(1+|x|^p),\\
E(\sup\limits_{t\leq r\leq s}|X_{t,r}(x)-X_{t,r}(x')-(x-x')|^p)\leq
M_p(s-t)(|x-x'|^p),\\
E(\sup\limits_{t\leq r\leq s}|\Delta_h^i[X_{t,r}(x)-x]|^p)\leq
M_p(s-t),\\
E(\sup\limits_{t\leq r\leq s}|\Delta_h^i X_{t,r}(x)-\Delta_{h'}^i
X_{t,r}(x')|^p)\leq M_p(s-t)(|x-x'|^p+|h-h'|^p),
\end{array}$$
where $\Delta_h^ig(x)=\frac{1}{h}(g(x+he_i)-g(x))$, and
$(e_1,\cdots,e_d)$ is an orthonormal basis of $\R^d$.
\end{proposition}
It is also known that  the stochastic flow  solution of a continuous
SDE satisfies the homeomorphic property (see Bismut \cite{b}, Kunita
\cite{k2}, \cite{k3}). But this property fails for the solution of
SDE with jumps in general. P.-A. Meyer in \cite{M81} (Remark p.111),
gave a counterexample with the following exponential
equation:
$$ X_{0,t} (x)  = x + \int_0^t  X_{0,s-}  dZ_s
$$ where $ Z$ is semimartingale, $Z_0=0$, such that $Z$ has a jump of size $-1$ at some stopping time $\tau$, $\tau >0$ a.s.
Then all trajectories of $X$, starting at any initial value $x$,
become zero at $\tau$ and stay there after $\tau$. This may be seen
trivially by the explicit form of the solution given by  the
Dol\'eans-Dade exponential:
$$
X_{0,t} (x) = x \exp \big(Z_t - \frac{1}{2}[Z,Z]^c_t \big) \prod_{0
< s \leq t} \big( 1 + \Delta Z_s \big) e^{-\Delta Z_s}.
$$
In the general setting of non-linear SDE, at the jump time $\tau$,
the solution jumps from $ X_{0,\tau-} (x)$ to $ X_{0,\tau-} (x) +
\beta (X_{0,\tau-} (x))$. L\'eandre \cite{l} gave a necessary and
sufficient condition under which  the homeomorphic property is
preserved at the jump time, namely,  for each $e \in {\E}$, the maps
$ H_e \, : \, x \mapsto \, x + \beta(x,e)$ should be one to one and
onto.  One can read also  Fujiwara and Kunita
\cite{fk} and  Protter \cite{protter} for  more
details on the subject. Therefore,  we  assume additionally that, for
each $ e \in \E$,  \textit{the linkage operator}:
$$\textbf{(A5)} \quad  \;  \mathrm{H}_e \, :\, x \mapsto \, x + \beta(x,e) \mbox{ is a }
C^2\mbox{-diffeomorphism.}$$ We denote by
$\mathrm{H}_e^{-1}$ the inverse map of $ \mathrm{H}_e$, and set $
h(x,e) := x - \mathrm{H}_e^{-1}(x)$. We have  the following result
where the proof can be found in  \cite{k4} (Theorem 3.13, p.359):
\begin{proposition}\label{flow}
Assume the assumptions \textbf{(A4)} and \textbf{(A5)} hold. Then
$\{X_{t,s}(x); x \in \mathbb{R}^d \}$ is a $C^2$-diffeomorphism a.s.
stochastic flow. Moreover the inverse of the flow satisfies the
following backward SDE
\begin{equation}\label{inverse:flow}
\begin{split}
X_{t,s}^{-1}(y) &  = y - \int_t^s \widehat{b}(X_{r,s}^{-1}(y)) dr  -
\int_t^s \sigma (X_{r,s}^{-1} (y)) \overleftarrow{dW}_r - \int_t^s
\int_{\E}  {\beta} (X_{r,s}^{-1} (y),e) \widetilde{\mu}  (\overleftarrow{dr},de)\\
 & \quad  + \int_t^s \int_{\E} \widehat{\beta}(X_{r,s}^{-1} (y),e) {\mu} (\overleftarrow{dr},de). \\
\end{split}
\end{equation}
for any  $t<s$,  where
\begin{equation}
\label{drift:backward}
\begin{split}
\widehat{b}(x) =  b(x) -  \sum_{i,j}\frac{\partial \sigma^j (x)
}{\partial x_i}  \sigma^{ij} (x) \quad \mbox{and}\quad
\widehat{\beta} (x,e) =  \beta (x, e) -  h (x,e).
\end{split}
\end{equation}
\end{proposition}
The explicit form (\ref{inverse:flow}) will be used in the proof of
the equivalence of norms (Proposition
\ref{equivalence:normes}).\medskip

\begin{remark}
 In (\ref{inverse:flow}), the three terms $\int_t^s \sigma (X_{r,s}^{-1} (y)) \overleftarrow{dW}_r$,
  $ \int_t^s\int_{\E}  {\beta} (X_{r,s}^{-1} (y),e) \widetilde{\mu}
(\overleftarrow{dr},de)$ and \\$\int_t^s \int_{\E}
\widehat{\beta}(X_{r,s}^{-1} (y),e) {\mu} (\overleftarrow{dr},de)$
are backward It\^o integrals. We refer the readers to literature
\cite{k4} for the definition (\cite{k4} p. 358). For convenience,
we give the definition of the backward It\^o integral with respect
to a Brownian motion. Let $f(r)$ be a right continuous backward
adapted process, then the backward It\^o integral is defined by
$$\int_t^s f(r) \overleftarrow{dW}_r:=\lim_{|\Pi|\rightarrow
0}\sum_{k} f(t_{k+1})(W_{t_{k+1}}-W_{t_k}),$$ where
$\Pi=\{t=t_0<t_1<\cdots<t_n=s\}$ are partitions of the interval
$[t,s]$. The other two terms can be defined similarly. Note that the inverse flow $X^{-1}_{r,s}$ is backward adapted, so
 we may define the backward integrals such as $\int_t^s \sigma (X_{r,s}^{-1} (y))
 \overleftarrow{dW}_r$ etc..
\end{remark}
\begin{remark}
In the paper of Ouknine and Turpin \cite{ouknineturpin06}, the
authors have weakened the regularity of the coefficients $b$ and
$\sigma$ of the diffusion, but added additional boundedness on them.
Since this improvement is not essential and the same discussion is
also valid for our case, we omit it.
\end{remark}\medskip

We denote by $J(X_{t,s}^{-1}(x))$ the determinant of the Jacobian
matrix of $X_{t,s}^{-1}(x)$, which is positive and
$J(X_{t,t}^{-1}(x))=1$. For $\phi\in C_c^{\infty}(\R^d)$, we define
a process $\phi_t:\,\Omega\times [t,T]\times \R^d\rightarrow \R$ by
\begin{equation}
\label{random:testfunction}
\phi_t(s,x):=\phi(X_{t,s}^{-1}(x))J(X_{t,s}^{-1}(x)).\end{equation}

We know that for $v\in \mathbf{L}^2(\R^d)$, the composition of $v$
with the stochastic flow is
$$(v\circ X_{t,s} (\cdot), \phi):=(v,\phi_t(s,\cdot)).$$

In fact, by a change of variable, we have $$(v\circ X_{t,s} (\cdot),
\phi)=\int_{\R^d}v(X_{t,s}
(x))\phi(x)dx=\int_{\R^d}v(y)\phi(X_{t,s}^{-1}(y))J(X_{t,s}^{-1}(y))dy
=(v,\phi_t(s,\cdot)).$$ Since  $ (\phi_t(s,x))_{ t\leq s}$ is a
process,  we may not use it directly as a test function because \\
$\int_t^T(u(s,\cdot),\partial_s\phi_t(s,\cdot))ds$ has no sense. However
$\phi_t(s,x)$ is a semimartingale and we have the following
decomposition of $\phi_t(s,x)$:
\begin{lemma}\label{decomposition}
For every function $\phi\in C_c^{\infty}(\R^d),$
\begin{equation}\label{decomp}\begin{array}{ll}
\phi_t(s,x)&=\phi(x)+\displaystyle\int_t^s{\mathcal
L}^\ast\phi_t(r,x)dr-\sum_{j=1}^{d}\int_t^s\left(\sum_{i=1}^{d}\frac{\partial}{\partial
x_i}(\sigma^{ij}(x)\phi_t(r,x))\right)dW_r^j\\&+\displaystyle\int_t^s\int_{\E}
{\mathcal A}_e ^\ast \phi_t(r-,x)\dtmur,
 \end{array} \end{equation}
where  ${\mathcal A}_e u(t,x)=u(t, \mathrm{H}_e (x))-u(t,x)$,
${\mathcal A}^\ast_e u(t,x)=u(t,\mathrm{H}^{-1}_e (x))
J(\mathrm{H}^{-1}_e (x))-u(t,x)$ and  ${\mathcal L}^\ast$ is the
adjoint operator of ${\mathcal L}$.
\end{lemma}
\begin{proof}  Assume that $v\in C_c^{\infty}(\R^d)$. Applying
the
change of variable $y=X_{t,s}^{-1}(x)$, we can get\\
$\begin{array}{ll} \integ{\R^d}{}v(x)(\phi_t(s,x)-\phi(x))dx&=
\integ{\R^d}{}v(x)(\phi(X_{t,s}^{-1}(x))J(X_{t,s}^{-1}(x))-\phi(X_{t,t}^{-1}(x))J(X_{t,t}^{-1}(x))dx\\
&=\integ{\R^d}{}(v(X_{t,s}(y))\phi(y)-v(y)\phi(y))dy\\
&=\integ{\R^d}{}\phi(y)(v(X_{t,s}(y))-v(y))dy.
\end{array}$\\
As $v$ is smooth enough, using It\^{o}'s formula for $
v(X_{t,s}(y))$, we have\\
$\begin{array}{ll} v(X_{t,s}(y))-v(y)&=\integ{t}{s}{\mathcal{L}}
v(X_{t,r-}(y))dr+\integ{t}{s}\sum_{i=1}^{d}\frac{\partial
v}{\partial
x_i}(X_{t,r}(y))\sum_{j=1}^{d}\sigma^{ij}(X_{t,r}(y))dW_r^j\\
&+\integ{t}{s}\int_{\E}{\mathcal A}_e v(X_{t,r-}(y))\dtmur.
\end{array}$\\
 Therefore,
$$\begin{array}{ll}
&\integ{\R^d}{}v(x)(\phi_t(s,x)-\phi(x))dx\\
=&\integ{\R^d}{}\phi(y)\Big\{\integ{t}{s}\mathcal{L}
v(X_{t,r}(y))dr+\integ{t}{s}\sum_{i=1}^{d}\frac{\partial v}{\partial
x_i}(X_{t,r}(y))\sum_{j=1}^{d}\sigma^{ij}(X_{t,r}(y))dW_r^j\Big\}dy\\
+&\integ{\R^d}{}\phi(y)\integ{t}{s}\int_{\E}{\mathcal A}_e
v(X_{t,r-}(y))\dtmur dy.
\end{array}$$
Since the first term has been dealt in \cite{bm}, and the adjoint
operator of ${\mathcal L}$ exists thanks to \cite{MPS04}, we focus
only on the second term. Using the stochastic Fubini theorem and
the change of variable $x=X_{t,r}(y)$  we obtain
\begin{equation*}
\begin{split}
\integ{\R^d}{}\phi(y)\integ{t}{s}\int_{\E}{\mathcal A}_e v(X_{t,r-}(y))\dtmur dy
&= \integ{t}{s}\int_{\E}\integ{\R^d}{}\phi(y){\mathcal A}_e v(X_{t,r-}(y))dy\dtmur\\
&=\integ{t}{s}\int_{\E}\integ{\R^d}{}\phi_t(r-,x){\mathcal A}_e
v(x)dx\dtmur.
\end{split}
\end{equation*}

Finally, we  use the change of variable $y = \mathrm{H}_e^{-1} (x)$
in the right hand side of the previous expression \beqn
&&\integ{t}{s}\int_{\E}\integ{\R^d}{}\phi_t(r-,x){\mathcal A}_e
v(x)dx\dtmur
\\&=& \integ{t}{s}\int_{\E}\integ{\R^d}{}\phi_t(r-,x) \big( v(\mathrm{H}_e (x))-v(x) \big) \,
dx\dtmur\\
&=&  \integ{t}{s}\int_{\E}\integ{\R^d}{}v(x){\mathcal
A}^\ast_e\phi_t(r-,x)dx\dtmur. \eeqn Since $v$ is an arbitrary
function, the lemma is proved.
\end{proof}
\vspace*{0.6cm}

We also  need  equivalence of norms result  which plays
an important role in the proof of the existence of the solution for
PIDE as a connection between the functional norms and random norms.
For continuous SDEs, this result  was first proved by Barles and Lesigne \cite{bl} by
using an analytic method. In \cite{bm}, the authors have proved the
result with a probabilistic method.  Note that Klimisiak \cite{Klimsiak} have extended this estimes for   Markov process associated  to a non-homogenious divergence  operator. The following result  generalize Proposition 5.1 in \cite{bm}  (see also \cite{bl}) in the case of a diffusion process with jumps, and the proof will be given in Appendix \ref{appendix:equivalencenorm}.
\begin{proposition}\label{equivalence:normes}
There exists two constants $c>0$ and $C>0$ such that for every
$t\leq s\leq T$ and $\varphi\in L^1(\R^d,dx)$,
\begin{equation}\label{equi1} c\int_{\R^d}|\varphi(x)|\rho(x)dx\leq
\int_{\R^d}E(|\varphi(X_{t,s}(x))|)\rho(x)dx\leq
C\int_{\R^d}|\varphi(x)|\rho(x)dx. 
\end{equation} 
Moreover, for
every $\Psi\in L^1([0,T]\times\R^d,dt\otimes dx)$,

\begin{equation} \label{equi2}
c\int_{\R^d}\int_t^T|\Psi(s,x)|ds\rho(x)dx \leq
\int_{\R^d}\int_t^TE(|\Psi(s,X_{t,s}(x))|)ds\rho(x)dx\leq
C\int_{\R^d}\int_t^T|\Psi(s,x)|ds\rho(x)dx.
\end{equation}
 
\end{proposition}

 We give now the following result which allows us to link by a natural way the solution of PIDE with the associated BSDE.
 Roughly speaking, if we choose in the variational formulation \eqref{wspde1} the random functions $\phi_t(\cdot,\cdot)$
 defined by \eqref{random:testfunction}, as a test functions, then we obtain the  associated BSDE.
 In fact, this result plays the same role as It\^o's formula used in \cite{pp1} and \cite{p1} (see \cite{pp1}, Theorem 3.1, p. 20)
 to relate the solution of some semilinear PDE's with the associated BSDE:
\begin{proposition}
\label{weak:Itoformula} Assume that all the previous assumptions
hold. Let $u\in {\mathcal H_T}$ be a weak solution of
PIDE(\ref{pde1}), then for $s\in[t,T]$ and $\phi\in
C_c^{\infty}(\R^d)$, \begin{equation}\label{wspde2}
\begin{array}{ll}
\displaystyle\int_{\R^d}\int_s^Tu(r,x)d\phi_t(r,x)dx+(u(s,x),\phi_t(s,x))-(g(x),\phi_t(T,x))
-\int_s^T( u(r,\cdot),\mathcal L^\ast \phi_t(r,\cdot))dr\\=\displaystyle\int_{\R^d}\int_s^Tf(r,x,u(r,x),\sigma^\ast\nabla
u(r,x),u(r,x+\beta(x,\cdot))-u(r,x))\phi_t(r,x)drdx. \quad a.s.
\end{array}
\end{equation}
where  $\int_{\R^d}\int_s^Tu(r,x)d\phi_t(r,x)dx $ is well defined
thanks to the semimartingale decomposition result (Lemma
\ref{decomposition}).
\end{proposition}

\begin{remark}
 Note that $\phi_t(r,x)$ is $\R$-valued. We consider that in
 (\ref{wspde2}), the equality holds for each component of $u$.
\end{remark}

\begin{remark}
 Under Brownian framework, this proposition is first proved by Bally
 and Matoussi in \cite{bm} for linear case via the polygonal approximation
 for Brownian motion (see Appendix A in \cite{bm} for more details). In fact,
 thanks especially to the fact that $\lambda$ is finite, we can make the similar approximation for It\^{o}-L\'{e}vy processes
 only by approximating polygonally the Brownian motion, the proof of this proposition follows step by step  the proof of Proposition 2.3 in \cite{bm} (pp. 156), so we omit it.
\end{remark}

\section{Sobolev solutions for parabolic semilinear PIDEs}
\label{section:PIDE}
In this section, we consider the PIDE (\ref{pde1}) under assumptions
(A1)-(A5). Moreover, we consider the following  decoupled  forward
backward stochastic differential equation (FBSDE in short) :
\beq\label{fbsde}
  \noindent\left\{
     \begin{array}{ll}
       X_{t,s}(x)= x+\integ{t}{s}b(X_{t,r}(x))dr+\integ{t}{s}\sigma(X_{t,r}(x))dW_r+\integ{t}{s}\int_{\E}
  \beta(X_{t,r-}(x),e)\dtmur;\\
       Y_s^{t,x}=g(X_{t,T}(x))+\integ{s}{T}f(r,X_{t,r}(x),Y_r^{t,x},Z_r^{t,x},V_r^{t,x})dr-\integ{s}{T}Z_r^{t,x}dW_r\\\hspace{5.5cm}
       -\integ{s}{T} \int_{\E} V_r^{t,x}(e)\dtmur.
     \end{array}
  \right.
\eeq 
According to Proposition 5.4 in \cite{cm} which deals with of
reflected BSDE, we know that (\ref{fbsde}) has a unique solution.
Moreover, we have the following estimate of the solution.
\begin{proposition}\label{prop:estimate} There exists a constant $c >0$  such that, for any $s\in[t,T]$:
 \begin{equation}\label{estimate}
   \sup\limits_{s\in[t,T]}E\big[\|Y_s^{t,\cdot}\|_2^2\big]+E\Big[\integ{t}{T}\|Z_s^{t,\cdot}\|_2^2ds+
   \integ{t}{T}\int_{\E} |V_s^{t,\cdot}(e)|_2^2\lambda(de)ds\Big]\leq
   c \big[ \|g \|_2^2 + \int_t^T \| f_s^0  \|_2^2 ds \big] .
 \end{equation}
\end{proposition}
Our main  result in this section is the following where the proof will be given in Appendix \ref{appendix:mainPIDE}:
\begin{theorem}
\label{main} Assume that \textbf{(A1)-(A5)} hold. There exists a
unique solution $u\in \mathcal H_T$ of the PIDE (\ref{pde1}).
Moreover, we have the probabilistic representation of the solution:
$u(t,x)=Y_t^{t,x}$, where $(Y_s^{t,x},Z_s^{t,x},V_s^{t,x})$ is the
solution of BSDE (\ref{fbsde}) and, we have $ds\otimes
d\P\otimes\rho(x)dx-a.e.$,
 \begin{equation}
 \label{representation}
  \begin{split}
  &Y_s^{t,x}=u(s,X_{t,s}(x)),\quad Z_s^{t,x}=(\sigma^\ast\nabla
  u)(s,X_{t,s}(x)),\\
  & V_s^{t,x}(\cdot)=u(s,X_{t,s-}(x)+\beta(X_{t,s-}(x),\cdot))-u(s,X_{t,s-}(x)).
\end{split}
 \end{equation}
\end{theorem}

\begin{remark}
 Since $u\in \mathcal H_T$, $u$ and $v=(\sigma^\ast\nabla
  u)$ are elements in $L^2_{\rho}([0,T]\times\mathbb{R}^d)$ and they are determined
  $\rho(x)dx$ a.e., but because of the equivalence of norms, there
  is no ambiguity in the definition of $u(s,X_{t,s}(x))$ and the others terms of \eqref{representation}.

\end{remark}

\begin{remark}
This stochastic flow method can be generalized to the study of
Sobolev solution of stochastic partial integro-differential
equations (SPIDEs for short) without essential difficulties (see
e.g. \cite{bm} for Brownian framewrok). More precisely, as the
authors have done in \cite{bm,pp1994}, by introducing an appropriate
backward doubly stochastic differential equation (BDSDE for short)
with jumps, we can provide a probabilistic interpretation for Sobolev
solution of an SPIDE by the solution of the BDSDE with jumps.
\end{remark}

\section{Obstacle problem for  PIDEs}
\label{section:obstacle}
In this part, we will study the obstacle
problem   (\ref{OPDIE})  with obstacle function
$h$, where we restrict our study in the one dimensional case ($n=1$).  We shall assume  the following hypothesis on the obstacle: \\[0.2cm]
\textbf{(A6)} $h\in C([0,T]\times\mathbb {R}^d;\mathbb {R})$ and  there exit $\iota,\ \kappa >0$ such that $|h(t,x)|\leq \iota (1+|x|^{\kappa}) $, for all $ x \in \mathbb R^d$.\\

We first introduce the reflected BSDE with jumps (RBSDE with jumps for short) associated with $(g,f,h)$  which has been studied  by Hamad\`ene and Ouknine \cite{ho}:
\begin{equation}
\label{rbsde1}
 \left\lbrace
\begin{aligned}
& Y_{s}^{t,x}
 =g(X_{t,T}(x))+
\int_{s}^{T}f(r,X_{t,r}(x),Y_{r}^{t,x},Z_{r}^{t,x},V_{r}^{t,x})dr+K_{T}^{t,x}-K_{s}^{t,x}\\
&\hspace{1.5cm} -\int_{s}^{T}Z_{r}^{t,x}dW_{r}-\integ{s}{T} \int_{\E} V_r^{t,x}(e)\dtmur,\;
P\text{-}a.s. , \; \forall \,  s \in [t,T]  \\
& Y_{s}^{t,x} \geq L_{s}^{t,x}, \quad \int_{t}^{T}(Y_{s}^{t,x}-L_{s}^{t,x})dK_{s}^{t,x}=0, \, \, P\text{-}a.s.\\
\end{aligned}
\right.
\end{equation}
The obstacle process  $L_{s}^{t,x}=h(s,X_{t,s}(x))$ is a c\`adl\`ag process which has only inaccessible jumps since $h$ is continuous and $(X_{t,s}(x))_{t\leq s\leq T}$ admits inaccessible jumps.  Moreover, using assumption \textbf{(A1)}  and \textbf{(A2)} and  equivalence of norm results 
(\ref{equi1}) and (\ref{equi2}), we get
\[
g(X_{t,T}(x)) \in \mathbf{L}^2({\cal F}_T),  \mbox{ and }
f(s,X_{t,s}(x),0,0,0) \in \mathcal {H}_{d}^2(t,T).
\]
Therefore according to  \cite{ho},  there exists a unique quadruple $(Y^{t,x},Z^{t,x},V^{t,x},K^{t,x})\in \mathcal {S}%
^{2}(t,T)\times \mathcal {H}_{d}^{2}(t,T)\times \mathcal {L}^{2}(t,T) \times\mathcal {A}^{2}(t,T)$
   solution of the
RBSDE with jumps \eqref{rbsde1}. \\[0.2cm]
More precisely, we consider  the  following definition  of weak solutions for the obstacle problem \eqref{OPDIE}:
\begin{definition}
\label{o-pde}We say that $(u,\nu )$ is the weak solution of the PIDE with
obstacle associated to $(g,f,h)$, if\\
(i) $\left\| u\right\|_{{\mathcal H}_T} ^{2}<\infty $, $u\geq h$, and $u(T,x)=g(x)$,\\
(ii) $\nu $ is a positive Radon  \textit{regular measure}  in the following sense, i.e.   for every measurable bounded and positive
functions $\phi $ and $\psi $,
\begin{align}
\nonumber &\int_{\mathbb{R}^{d}}\int_{t}^{T}\phi (s,X^{-1}_{t,s}(x))J(X^{-1}%
_{t,s}(x))\psi (s,x)1_{\{u=h\}}(s,x)\nu (ds,dx)\\
&=\int_{\mathbb{R}%
^{d}}\int_{t}^{T}\phi (s,x)\psi (s,X_{t,s}(x))dK_{s}^{t,x}dx\text{, a.s..}
\label{con-k}
\end{align}
where $(Y_{s}^{t,x},Z_{s}^{t,x},V_{s}^{t,x},K_{s}^{t,x})_{t\leq s\leq T}$ is the
solution of RBSDE with jumps (\ref{rbsde1}) and  such that \\$\int_{0}^{T}\int_{\mathbb{R}
^{d}}\rho (x)\nu (dt,dx)<\infty ,$\\
(iii) for every $\phi \in \mathcal D_T$%
\begin{align}\label{OPDE}
\nonumber &\int_{t}^{T}\int_{\mathbb{R}^{d}}u(s,x)\partial _{s}\phi(s,x)dxds+\int_{\mathbb{R}^{d}}(u(t,x )\phi (t,x
)-g(x )\phi (T,x))dx+\int_{t}^{T}\int_{\mathbb{R}^{d}}u(s,x)\mathcal{L}^*\phi(s,x)dxds\\
\nonumber &=\int_{t}^{T}\int_{\mathbb{R}^{d}}f(s,x ,u,\sigma ^{*}\nabla u, u(s,x+\beta(x,\cdot))-u(s,x))\phi(s,x)dxds
 +\int_{t}^{T}\int_{\mathbb{R}^{d}}\phi (s,x)1_{\{u=h\}}(s,x)\nu (ds,dx). \\
\end{align}
\end{definition}

First, we give a  weak It\^o's formula  similar to  the one given in Proposition \ref{weak:Itoformula}. This result is essential to show the link between a Sobolev solution to the obstacle problem and the associated reflected BSDE with jumps, which in turn insures the uniqueness of the solution. The proof of this proposition is the same as Proposition \ref{weak:Itoformula}.
\begin{proposition}
\label{weak:Itoformula1} Assume that conditions  \textbf{(A1)}-\textbf{(A6)} hold and  $\rho (x)=(1+\left| x\right|)^{-p}$ with $p\geq \gamma $ where
$\gamma =\kappa +d+1$. Let $u\in {\mathcal H_T}$ be a weak solution of
PIDEs with obstacle associated to $(g,f,h)$, 
then for $s\in[t,T]$ and $\phi\in
C_c^{\infty}(\R^d)$, \begin{equation}
\begin{array}{ll}
\displaystyle\int_{\R^d}\int_s^Tu(r,x)d\phi_t(r,x)dx+(u(s,\cdot),\phi_t(s,\cdot))-(g(\cdot),\phi_t(T,\cdot))
-\int_s^T( u(r,\cdot),\mathcal L^\ast \phi_t(r,\cdot))dr\\=\displaystyle\int_{\R^d}\int_s^Tf(r,x,u(r,x),\sigma^\ast\nabla
u(r,x),u(r,x+\beta(x,\cdot))-u(r,x))\phi_t(r,x)drdx\\+\displaystyle\int_{\mathbb{%
R}^{d}}\int_{s}^{T}\phi _{t}(r,x)1_{\{u=h\}}(r,x)\nu (dr,dx). \quad a.s.
\end{array}
\end{equation}
where  $\int_{\R^d}\int_s^Tu(r,x)d\phi_t(r,x)dx $ is well defined
thanks to the semimartingale decomposition result (Lemma
\ref{decomposition}).
\end{proposition}

The main result of this section is the following

\begin{theorem}
\label{mr2} Assume that conditions  \textbf{(A1)}-\textbf{(A6)} hold
and $\rho (x)=(1+\left| x\right|)^{-p}$ with $p\geq \gamma $ where
$\gamma =\kappa +d+1$.
There exists a weak solution $(u,\nu
)$ of the PIDE with obstacle (\ref{OPDIE})
associated to $(g,f,h)$ such that, $ds\otimes
d\P\otimes\rho(x)dx-a.e.$,
\begin{align}\label{con-pre}
\nonumber &Y_{s}^{t,x}=u(s,X_{t,s}(x)),Z_{s}^{t,x}=(\sigma ^{*}\nabla
u)(s,X_{t,s}(x)),\\
&V_{s}^{t,x}(\cdot)=u(s,X_{t,s-}(x)+\beta (X_{t,s-}(x),\cdot))-u(s,X_{t,s-}(x))\text{, a.s..} 
\end{align}
Moreover,  the reflected measure $\nu$ is a regular measure in the sense of the definition (ii) and satisfying the probabilistic interpretation \eqref{con-k}. 

If $(\overline{u},\overline{\nu })$ is another solution of the PIDE with obstacle(\ref
{OPDIE}) such that $\overline{\nu }$ satisfies (\ref{con-k}) with some $%
\overline{K}$ instead of $K$, where $\overline{K}$ is a continuous process
in $\mathcal {A}^{2}(t,T)$, then $\overline{u}=u$ and $\overline{%
\nu }=\nu $.

In other words, there is a unique Randon regular measure with support $\{u=h\}$ which satisfies (\ref{con-k}).
\end{theorem}

\begin{remark}
The expression (\ref{con-k}) gives us the probabilistic
interpretation (Feymamn-Kac's formula) for the measure $\nu $ via
the nondecreasing process $K^{t,x}$ of the RBSDE with jumps. This formula was
first introduced in Bally et al. \cite{BCEF} (see also \cite{MX08}). Here we generalize their results to the case of PIDEs.
\end{remark}

From Lemma 3.1 in \cite{cm}, we know that if we have more information on the obstacle $L$, we can give a more explicit representation for the processes $K$. Then as a result of the above theorem, we have when $h$ is smooth enough, the reflected measure $\nu$ is Lebesgue absolute continuous, moreover there exist a unique $\widetilde{\nu}$ and a measurable function $(\alpha_s)_{s\geq 0}$ such that $\nu(ds,dx)=\alpha_s\widetilde{\nu}_s(dx)ds$ 
.

\vspace{0.5em}
\begin{proof}
\textbf{ a) Existence}: The existence of a solution
will be proved in two steps. For the first step, we suppose that
$f$ does not depend on $y,z,w$, then we are able to apply the usual penalization method. 
In the second step, we study the case when $f$ depends on $y,z,w$ with the result obtained in the first step.\\[0.1cm] 

\textit{Step 1} :
We will use the penalization method. For $n\in \mathbb{N}$, we consider for
all $s\in [t,T]$,
\begin{align*}
Y_{s}^{n,t,x}=g(X_{t,T}(x))&+\int_{s}^{T}f(r,X_{t,r}(x))dr+n%
\int_{s}^{T}(Y_{r}^{n,t,x}-h(r,X_{t,r}(x)))^{-}dr\\
&-\int_{s}^{T}Z_{r}^{n,t,x}dW_{r}-\integ{s}{T} \int_{\E} V_r^{n,t,x}(e)\dtmur.
\end{align*}

From Theorem (\ref{main}) in section 3, we know that $u_{n}(t,x):=Y_{t}^{n,t,x}$%
, is solution of the PIDE$(g,f_{n})$, where $%
f_{n}(t,x,y)=f(t,x)+n(y-h(t,x))^{-}$, i.e. for every $\phi \in\mathcal{D}_T$%
\begin{align}\label{o-equa1}
 \nonumber\int_{t}^{T}(u^{n}(s,\cdot),\partial
_{s}\phi(s,\cdot) )ds & +(u^{n}(t,\cdot ),\phi
(t,\cdot ))-(g(\cdot ),\phi (T,\cdot))+\int_{t}^{T}%
(u^{n}(s,\cdot),\mathcal{L}^*\phi(s,\cdot))ds\\
&=\int_{t}^{T}(f(s,\cdot),\phi(s,\cdot))ds+n\int_{t}^{T}((u^{n}-h)^{-}(s,\cdot ),\phi(s,\cdot))ds.
\end{align}
Moreover
\begin{align}\label{rep1}
\nonumber &Y_{s}^{n,t,x}=u_{n}(s,X_{t,s}(x)),Z_{s}^{n,t,x}=\sigma ^{*}\nabla
u_{n}(s,X_{t,s}(x)),\\
&V_{s}^{n,t,x}(\cdot)=u_{n}(s,X_{t,s-}(x)+\beta (X_{t,s-}(x),\cdot))-u_{n}(s,X_{t,s-}(x)) 
\end{align}
Set $K_{s}^{n,t,x}=n \displaystyle
\int_{t}^{s}(Y_{r}^{n,t,x}-h(r,X_{t,r}(x)))^{-}dr$. Then by
(\ref{rep1}), we have that $K_{s}^{n,t,x}=n \displaystyle
\int_{t}^{s}(u_{n}-h)^{-}(r,X_{t,r}(x))dr$.

Following the estimates and convergence results for
$(Y^{n,t,x},Z^{n,t,x},V^{n,t,x},K^{n,t,x})$ in the step 3 and step 5 of the proof of Theorem 1.2.  in
\cite{ho},  we get as $m$, $n$ tend to infinity  :

\begin{align*}
& E\sup_{t\leq s\leq T}\left| Y_{s}^{n,t,x}-Y_{s}^{m,t,x}\right|
^{2}+E\int_{t}^{T}\left| Z_{s}^{n,t,x}-Z_{s}^{m,t,x}\right|
^{2}ds\\
&+E\int_{t}^{T}\int_{\mathbb E}\left| V_{s}^{n,t,x}(e)-V_{s}^{m,t,x}(e)\right|
^{2}\lambda(de)ds+E\sup_{t\leq s\leq T}\left|
K_{s}^{n,t,x}-K_{s}^{m,t,x}\right| ^{2}\longrightarrow 0,
\end{align*}
and

\begin{align}
\label{K-estimate}
\nonumber &\sup_n E\left[\sup_{t\leq s\leq T}\left| Y_{s}^{n,t,x}\right| ^{2}+\int_{t}^{T}(\left|
Z_{s}^{n,t,x}\right|^2ds)+\int_{t}^{T}\int_{\mathbb E}\left| V_{s}^{n,t,x}(e)\right|^{2}\lambda(de)ds + (K_{T}^{n,t,x})^{2}\right]\\
&\leq C\left( 1+\lvert x\rvert^{2\kappa}\right) .
\end{align}

By the equivalence of norms (\ref{equi2}), we get
\begin{eqnarray*}
&&\int_{\mathbb{R}^{d}}\int_{t}^{T}\rho (x)(\left|
u_{n}(s,x)-u_{m}(s,x)\right| ^{2}+\left| \sigma ^{*}\nabla u_{n}(s,x)-\sigma
^{*}\nabla u_{m}(s,x)\right| ^{2})dsdx \\
&\leq &\frac{1}{k_{2}}\int_{\mathbb{R}^{d}}\rho (x)E\int_{t}^{T}(\left|
Y_{s}^{n,t,x}-Y_{s}^{m,t,x}\right| ^{2}+\left|
Z_{s}^{n,t,x}-Z_{s}^{m,t,x}\right| ^{2})dsdx\rightarrow 0.
\end{eqnarray*}
Thus $(u_{n})_{n\in\mathbb N}$ is a Cauchy sequence in $\mathcal{H}_T$, and the limit $%
u=\lim_{n\rightarrow \infty }u_{n}$ belongs to  $\mathcal{H}_T$.
Denote $\nu _{n}(dt,dx)=n(u_{n}-h)^{-}(t,x)dtdx$ and $\pi _{n}(dt,dx)=\rho
(x)\nu _{n}(dt,dx)$, then by (\ref{equi2})
\begin{eqnarray*}
\pi _{n}([0,T]\times \mathbb{R}^{d}) &=&\int_{\mathbb{R}^{d}}\int_{0}^{T}%
\rho (x)\nu _{n}(dt,dx)=\int_{\mathbb{R}^{d}}\int_{0}^{T}\rho
(x)n(u_{n}-h)^{-}(t,x)dtdx \\
&\leq &\frac{1}{k_{2}}\int_{\mathbb{R}^{d}}\rho (x)E\left|
K_{T}^{n,0,x}\right| dx\leq C\int_{\mathbb{R}^{d}}\rho (x)\left( 1+\lvert x\rvert^{\kappa}\right)dx<\infty .
\end{eqnarray*}
It follows that
\begin{equation}
\sup_{n}\pi _{n}([0,T]\times \mathbb{R}^{d})<\infty .  \label{est-measure}
\end{equation}
Moreover  by Lemma \ref{tight} (see Appendix \ref{appendix:tight}), the sequence of measures $(\pi _{n})_{n  \in \mathbb N}$  is tight.  Therefore, there exits a subsequence such that $(\pi _{n})_{n  \in \mathbb N}$  converges  weakly to a positive measure $\pi $.\\ Define $\nu =\rho
^{-1}\pi $; $\nu $ is a positive measure such that $\int_{0}^{T}\int_{%
\mathbb{R}^{d}}\rho (x)\nu (dt,dx)<\infty $, and so we have for
$\phi \in \mathcal{D}_T$ with
compact support in $x$,
\[
\int_{\mathbb{R}^{d}}\int_{t}^{T}\phi d\nu _{n}=\int_{\mathbb{R}%
^{d}}\int_{t}^{T}\frac{\phi }{\rho }d\pi _{n}\rightarrow \int_{\mathbb{R}%
^{d}}\int_{t}^{T}\frac{\phi }{\rho }d\pi =\int_{\mathbb{R}%
^{d}}\int_{t}^{T}\phi d\nu .
\]

Now passing to the limit in the PIDEs $(g,f_{n})$ (\ref{o-equa1}), we get that  that $(u,\nu )$
satisfies the PIDsE with obstacle associated to  $(g,f,h)$, i.e. for every $\phi \in
\mathcal{D}_T$, we have
\begin{eqnarray}
&&\int_{t}^{T}(u(s,\cdot),\partial _{s}\phi(s,\cdot) )ds+(u(t,\cdot ),\phi (t,\cdot))-(g(\cdot),\phi (T,\cdot))+\int_{t}^{T}(u(s,\cdot),\mathcal{L}^*\phi(s,\cdot))ds
\nonumber \\
&=&\int_{t}^{T}(f(s,\cdot),\phi(s,\cdot) )ds+\int_{t}^{T}\int_{\mathbb{R}%
^{d}}\phi (s,x)\nu (ds,dx).  \label{equa1} 
\end{eqnarray}
The last point is to prove that $\nu $ satisfies the probabilistic interpretation (%
\ref{con-k}). Since $K^{n,t,x}$ converges to $K^{t,x}$ uniformly in $t$, the
measure $dK^{n,t,x}\rightarrow dK^{t,x}$ weakly in probability.\\
Fix two continuous functions $\phi $, $\psi $ : $[0,T]\times \mathbb{R}%
^{d}\rightarrow \mathbb{R}^{+}$ which have compact support in $x$ and a
continuous function with compact support $\theta :\mathbb{R}^{d}\rightarrow %
\mathbb{R}^{+}$, we have
\begin{eqnarray*}
&&\int_{\mathbb{R}^{d}}\int_{t}^{T}\phi (s,X^{-1}_{t,s}(x))J(X%
^{-1}_{t,s}(x))\psi (s,x)\theta (x)\nu (ds,dx) \\
&=&\lim_{n\rightarrow \infty }\int_{\mathbb{R}^{d}}\int_{t}^{T}\phi (s,%
X^{-1}_{t,s}(x))J(X^{-1}_{t,s}(x))\psi (s,x)\theta
(x)n(u_{n}-h)^{-}(s,x)dsdx \\
&=&\lim_{n\rightarrow \infty }\int_{\mathbb{R}^{d}}\int_{t}^{T}\phi
(s,x)\psi (s,X_{t,s}(x))\theta
(X_{t,s}(x))n(u_{n}-h)^{-}(t,X_{t,s}(x))dsdx \\
&=&\lim_{n\rightarrow \infty }\int_{\mathbb{R}^{d}}\int_{t}^{T}\phi
(s,x)\psi (s,X_{t,s}(x))\theta (X_{t,s}(x))dK_{s}^{n,t,x}dx \\
&=&\int_{\mathbb{R}^{d}}\int_{t}^{T}\phi (s,x)\psi (s,X_{t,s}(x))\theta
(X_{t,s}(x))dK_{s}^{t,x}dx.
\end{eqnarray*}

We take $\theta =\theta _{R}$ to be the regularization of the indicator
function of the ball of radius $R$ and pass to the limit with $R\rightarrow
\infty $, it follows that
\begin{equation}\label{con-k1}
\int_{\mathbb{R}^{d}}\int_{t}^{T}\phi (s,X^{-1}_{t,s}(x))J(X^{-1}%
_{t,s}(x))\psi (s,x)\nu (ds,dx)=\int_{\mathbb{R}^{d}}\int_{t}^{T}\phi
(s,x)\psi (s,X_{t,s}(x))dK_{s}^{t,x}dx.
\end{equation}
Since $(Y_{s}^{n,t,x},Z_{s}^{n,t,x},V_{s}^{n,t,x},K_{s}^{n,t,x})$ converges to $%
(Y_{s}^{t,x},Z_{s}^{t,x},V_{s}^{t,x},K_{s}^{t,x})$ as $n\rightarrow \infty $ in $\mathcal {%
S}^{2}(t,T)$ $\times \mathcal { H}_d^{2}(t,T)\times \mathcal {L}^{2}(t,T)\times \mathcal {A}^{2}(t,T)$, and $%
(Y_{s}^{t,x},Z_{s}^{t,x},V_{s}^{t,x},K_{s}^{t,x})$ is the solution of RBSDEs with jumps ($g(X_{t,T}(x))$, $f$, $h$), then we have
\[
\int_{t}^{T}(Y_{s}^{t,x}-L_{s}^{t,x})dK_{s}^{t,x}=%
\int_{t}^{T}(u-h)(s,X_{t,s}(x))dK_{s}^{t,x}=0,\text{a.s.}
\]
It follows that $dK_{s}^{t,x}=1_{\{u=h\}}(s,X_{t,s}(x))dK_{s}^{t,x}$. In (\ref{con-k1}), setting $\psi =1_{\{u=h\}}$ yields

\begin{align*}
&\int_{\mathbb{R}^{d}}\int_{t}^{T}\phi (s,X^{-1}_{t,s}(x))J(X^{-1}%
_{t,s}(x))1_{\{u=h\}}(s,x)\nu (ds,dx)\\&=\int_{\mathbb{R}^{d}}\int_{t}^{T}\phi
(s,X^{-1}_{t,s}(x))J(X^{-1}_{t,s}(x))\nu (ds,dx)\text{, a.s.}
\end{align*}

Note that the family of functions $A(\omega )=\{(s,x)\rightarrow \phi (s,%
X^{-1}_{t,s}(x)):\phi \in C_{c}^{\infty }\}$ is an algebra which
separates the points (because $x\rightarrow X^{-1}_{t,s}(x)$ is a
bijection). Given a compact set $G$, $A(\omega )$ is dense in $C([0,T]\times
G)$. It follows that $J(X^{-1}_{t,s}(x))1_{\{u=h\}}(s,x)\nu (ds,dx)=J(%
X^{-1}_{t,s}(x))\nu (ds,dx)$ for almost every $\omega $. While $J(%
X^{-1}_{t,s}(x))>0$ for almost every $\omega $, we get $\nu
(ds,dx)=1_{\{u=h\}}(s,x)\nu (ds,dx)$, and (\ref{con-k}) follows.
Then we get easily that $Y_{s}^{t,x}=u(s,X_{t,s}(x))$, $%
Z_{s}^{t,x}=\sigma ^{*}\nabla u(s,X_{t,s}(x))$ and $V_{s}^{t,x}(\cdot)=u(s,X_{t,s-}(x)+\beta (X_{t,s-}(x),\cdot))-u(s,X_{t,s-}(x))$,  in view of the convergence
results for $(Y_{s}^{n,t,x},Z_{s}^{n,t,x},V_{s}^{n,t,x})$ and the equivalence of norms. So $u(s,X_{t,s}(x))=Y_{s}^{t,x}\geq h(s,X_{t,s}(x))$. Specially for $s=t$, we
have $u(t,x)\geq h(t,x)$.\\[0.2cm]
\textit{ Step 2 } : \textit{The  nonlinear  case where $f$
depends on $y, z$  and $ w$}.\\  Let define $F(s,x)\triangleq
f(s,x,Y_s^{s,x},Z_s^{s,x},V_s^{s,x}).$  By plugging into the facts that
$f^0\in \mathbf{L}^2_\rho([0,T]\times\R^d)$ and $f$ is Lipschitz
with respect to $(y,z,v)$, then thanks to Proposition  \ref{prop:estimate}
we have $F(s,x)\in \mathbf{L}^2_\rho([0,T]\times\R^d)$. Since $F$ is independent of $y,z,w$, by applying the result of Step 1 yields that there exists $(u,\nu )$ satisfying the PIDEs with obstacle $(g,F,h)$, i.e. for every $\phi \in
\mathcal{D}_T$, we have
\begin{eqnarray}
&&\int_{t}^{T}(u(s,\cdot),\partial _{s}\phi(s,\cdot) )ds+(u(t,\cdot ),\phi (t,\cdot))-(g(\cdot),\phi (T,\cdot))+\int_{t}^{T}(u(s,\cdot),\mathcal{L}^*\phi(s,\cdot))ds
\nonumber \\
&=&\int_{t}^{T}(F(s,\cdot),\phi(s,\cdot) )ds+\int_{t}^{T}\int_{\mathbb{R}%
^{d}}\phi (s,x)1_{\{u=h\}}(s,x)\nu (ds,dx).  \label{equa2}
\end{eqnarray}
Then by the uniqueness of the solution to the RBSDEs with jumps ($g(X_{t,T}(x))$, $f$, $h(X_{t,s}(x))$),  we get easily that $Y_{s}^{t,x}=u(s,X_{t,s}(x))$, $%
Z_{s}^{t,x}=\sigma ^{*}\nabla u(s,X_{t,s}(x))$, $V_{s}^{t,x}(\cdot)=u(s,X_{t,s-}(x)+\beta (X_{t,s-}(x),\cdot))-u(s,X_{t,s-}(x))$, and $\nu$ satisfies the probabilistic interpretation (\ref{con-k}). So $u(s,X_{t,s}(x))=Y_{s}^{t,x}\geq h(s,X_{t,s}(x))$. Specially for $s=t$, we have $u(t,x)\geq h(t,x)$, which is the desired result.
\\[0.1cm]

\textbf{ b) Uniqueness } :  Set $(\overline{u},\overline{\nu
})$ to be another weak solution of the PIDEs with
obstacle (\ref{OPDE}) associated to $(g,f,h)$; with $\overline{\nu }$ verifies (%
\ref{con-k}) for a nondecreasing process $\overline{K}$. We fix $\phi :%
\mathbb{R}^{d}\rightarrow \mathbb{R}$, a smooth function in $C_{c}^{2}(%
\mathbb{R}^{d})$ with compact support and denote $\phi _{t}(s,x)=\phi (%
X^{-1}_{t,s}(x))J(X^{-1}_{t,s}(x))$. From Proposition \ref{weak:Itoformula1}%
, one may use $\phi _{t}(s,x)$ as a test function in the PIDEs $(g,f,h)$ with $%
\partial _{s}\phi (s,x)ds$ replaced by a stochastic integral with respect to
the semimartingale $\phi _{t}(s,x)$. Then we get, for $t\leq s\leq T$
\begin{align}
&\int_{s}^{T}\int_{\mathbb{R}^{d}}\overline{u}(r,x)d\phi _{t}(r,x)dx+(
\overline{u}(s,\cdot ),\phi _{t}(s,\cdot ))-(g(\cdot ),\phi _{t}(T,\cdot))-\int_{s}^{T}\int_{\mathbb{R}^{d}}\overline{u}(r,x)\mathcal{L}^*\phi_t(r,x)drdx \label{o-pde-u1}
\nonumber \\
&=\int_{s}^{T}\int_{\mathbb{R}^{d}}f(r,x,\overline{u}(r,x),\sigma
^{*}\nabla \overline{u}(r,x),\overline{u}(r,x+\beta(x,\cdot))-\overline{u}(r,x))\phi _{t}(r,x)drdx \nonumber \\
& +\int_{s}^{T}\int_{\mathbb{
R}^{d}}\phi _{t}(r,x)1_{\{\overline{u}=h\}}(r,x)\overline{\nu }(dr,dx).
\end{align}
By (\ref{decomp}) in Lemma \ref{decomposition}, we have
\begin{align*}
&\int_{s}^{T}\int_{\mathbb{R}^{d}}\overline{u}(r,x)d\phi _{t}(r,x)dx
=\int_{s}^{T}(\int_{\mathbb{R}^{d}}(\sigma ^{*}\nabla \overline{u}
)(r,x)\phi
_{t}(r,x)dx)dW_{r}\\&+\displaystyle\int_s^T\int_{\E}\int_{\mathbb{R}^{d}}
\overline{u}(r,x) {\mathcal A}_e^\ast\phi_t(r-,x)dx\dtmur
 +\int_{s}^{T}\int_{\mathbb{R}^{d}}\overline{u}(r,x)\mathcal{L}^*\phi_t(r,x)drdx.
\end{align*}
Substitute this equality in (\ref{o-pde-u1}), we get
\begin{align*}
&\int_{\mathbb{R}^{d}}\overline{u}(s,x)\phi _{t}(s,x)dx =(g(\cdot
),\phi _{t}(T,\cdot))-\int_{s}^{T}(\int_{\mathbb{R}^{d}}(\sigma
^{*}\nabla \overline{u})(r,x)\phi
_{t}(r,x)dx)dW_{r}\\&+\displaystyle\int_s^T\int_{\E}\int_{\mathbb{R}^{d}}
{\mathcal A}_e \overline{u}(r-,x) \phi_t(r,x)dx\dtmur \\
&+\int_{s}^{T}\int_{\mathbb{R}^{d}}f(r,x,\overline{u}(r,x),\sigma
^{*}\nabla \overline{u}(r,x),\overline{u}(r,x+\beta(x,\cdot))-\overline{u}(r,x))\phi _{t}(r,x )drdx\\&+\int_{s}^{T}\int_{\mathbb{%
R}^{d}}\phi _{t}(r,x)1_{\{\overline{u}=h\}}(r,x)\overline{\nu }(dr,dx).
\end{align*}
Then by changing of variable $y=X^{-1}_{t,r}(x)$ and applying (\ref
{con-k}) for $\overline{\nu }$, we obtain
\begin{align*}
&\int_{\mathbb{R}^{d}}\overline{u}(s,X_{t,s}(y))\phi (y)dy =\int_{\mathbb{R}^{d}}g(X_{t,T}(y))\phi (y)dy\\\
&+\int_{s}^{T}\int_{\mathbb{R}^{d}}\phi(y)f(r,X_{t,r}(y),\overline{u}(r,X_{t,r}(y)),\sigma ^{*}\nabla \overline{u}%
(r,X_{t,r}(y)),{\mathcal A}_e \overline{u}(r,X_{t,r-}(y)))drdy \\
&+\int_{s}^{T}\int_{\mathbb{R}^{d}}\phi (y)1_{\{\overline{u}%
=h\}}(r,X_{t,s}(y))d\overline{K}_{r}^{t,y}dy-\int_{s}^{T}(\int_{\mathbb{R}%
^{d}}(\sigma ^{*}\nabla \overline{u})(r,X_{t,r}(y))\phi (y)dy)dW_{r}\\&+\int_{s}^{T}\int_{\mathbb E}\left(\int_{\mathbb{R}^{d}}{\mathcal A}_e \overline{u}(r,X_{t,r-}(y))\phi(y)dy\right)\tilde{\mu}(dr,de).
\end{align*}
Since $\phi $ is arbitrary, we can prove that for $\rho (y)dy$ almost every $%
y$, ($\overline{u}(s,X_{t,s}(y))$, $(\sigma ^{*}\nabla \overline{u}
)(s,X_{t,s}(y))$, $\overline{u}(s,X_{t,s-}(y)+\beta(X_{t,s-}(y),\cdot))-\overline{u}(s,X_{t,s-}(y)),\widehat{K}_{s}^{t,y}$) solves the RBSDEs with jumps
$(g(X_{t,T}(y)),f,h)$. Here $\widehat{K}_{s}^{t,y}$=$\int_{t}^{s}1_{\{%
\overline{u}=h\}}(r,X_{t,r}(y))d\overline{K}_{r}^{t,y}$. Then by the
uniqueness of the solution of the RBSDEs with jumps, we know $\overline{u}%
(s,X_{t,s}(y))=Y_{s}^{t,y}=u(s,X_{t,s}(y))$, $(\sigma ^{*}\nabla \overline{%
u})(s,X_{t,s}(y))=Z_{s}^{t,y}=(\sigma ^{*}\nabla u)(s,X_{t,s}(y))$, $\overline{u}(s,X_{t,s-}(y)+\beta (X_{t,s-}(y),\cdot))-\overline{u}(s,X_{t,s-}(y))=V_{s}^{t,y}(\cdot)=u(s,X_{t,s-}(y)+\beta (X_{t,s-}(y),\cdot))-u(s,X_{t,s-}(y))$ and $%
\widehat{K}_{s}^{t,y}=K_{s}^{t,y}$. Taking $s=t$ we deduce that $\overline{u}%
(t,y)=u(t,y)$, $\rho (y)dy$-a.s. and by the probabilistic interpretation (%
\ref{con-k}), we obtain
\[
\int_{s}^{T}\int_{\R^d} \phi _{t}(r,x)1_{\{\overline{u}=h\}}(r,x)\overline{\nu }%
(dr,dx)=\int_{s}^{T}\int_{\R^d} \phi _{t}(r,x)1_{\{u=h\}}(r,x)\nu (dr,dx).
\]
So $1_{\{\overline{u}=h\}}(r,x)\overline{\nu
}(dr,dx)=1_{\{u=h\}}(r,x)\nu (dr,dx)$.
\end{proof}

\section{Appendix}
\label{Appendix}
\subsection{Proof of  Proposition \ref{equivalence:normes}}
\label{appendix:equivalencenorm}
  In order to prove (\ref{equi1}), it is sufficient
to prove that $$c\leq E[\frac{J(X^{-1}_{t,s}(x))\rho(
X^{-1}_{t,s}(x))}{\rho(x)}]\leq C.$$ In fact, making the change of
variable $y=X_{t,s}(x)$, we can get the following relation:
$$\int_{\R^d}E(|\varphi(X_{t,s}(x))|)\rho(x)dx=\int_{\R^d}|\varphi(y)|E[\frac{J(X^{-1}_{t,s}(y))\rho(X^{-1}_{t,s}(y))}{\rho(y)}]\rho(y)dy.$$

We differentiate with respect to $y$ in \eqref{inverse:flow} in
order to get:
\begin{equation}
\label{inverse:flow1}
\begin{array}{lll}
\nabla X_{t,s}^{-1}(y) &  = & \displaystyle I - \int_t^s \nabla
\widehat{b}(X_{r,s}^{-1}(y) \nabla X_{t,r}^{-1}(y)dr  -
\int_t^s \nabla \sigma (X_{r,s}^{-1} (y))\nabla X_{t,r}^{-1}(y) \overleftarrow{dW}_r\\
& & - \displaystyle\int_t^s \int_{\E}  \nabla {\beta} (X_{r,s}^{-1}
(y),e) \nabla
X_{t,r}^{-1}(y) \widetilde{\mu}  (\overleftarrow{dr},de)\\
&&+ \displaystyle\int_t^s \int_{\E} \nabla \widehat{\beta}(X_{r,s}^{-1} (y),e) \nabla X_{t,r}^{-1}(y) {\mu} (\overleftarrow{dr},de) \\
& :=& I + \Gamma_{t,s} (y)
\end{array}
\end{equation}
 where $\nabla b$, $\nabla \sigma$, $\nabla \beta$ and $\nabla \widehat{\beta}$ are the gradient of $b$, $\sigma$,$\beta$ and $\widehat{\beta}$, respectively  and $I$ is the
identity matrix.
Since $J(X_{t,s}^{-1}(y)):=\det \nabla
X_{t,s}^{-1}(y)=\inf_{\|\xi\|=1}< \nabla X^{-1}_{t,s}(y)\xi,\xi>$, we
obtain
$$1-\|\Gamma_{t,s} (y)\|\leq J(X_{t,s}^{-1}(y))\leq 1+\|\Gamma_{t,s} (y)\|.$$
Writing $\Gamma_{t,s} (y):=C_{t,s}(y)+D_{t,s}(y)$, where
$D_{t,s}(y)$ denotes the integration with respect to the random
measure and $C_{t,s}(y)$
denotes the others.\\
According to Bally and Matoussi \cite{bm}, we know that for any $y$,
$E[|C_{t,s}(y)|^2]\leq K(s-t)$. Now we are going to prove the
similar relation for $D_{t,s}(y)$. We only deal with the first term
of $D_{t,s}(y)$ because another one can be treated similarly without
any difficulty. In fact, by Burkholder-David-Gundy inequality (of
course in the backward sense), we have
$$\begin{array}{lll}
&&E[\displaystyle|\int_t^s \int_{\E}  \nabla {\beta} (X_{r,s}^{-1}
(y),e) \nabla X_{t,r}^{-1}(y) \widetilde{\mu}
(\overleftarrow{dr},de)|^2]\\
  &\leq& C E[\integ{t}{s}\int_{\E}|\nabla {\beta}(X^{-1}_{r,s}(y),e)\nabla X^{-1}_{r,s}(y)|^2\lambda(de)dr]\\
  &\leq& C_1E[\integ{t}{s}\int_{\E}(1\wedge|e|)^2|\nabla X^{-1}_{r,s}(y)|^2\lambda(de)dr]\\
  &=&C_1E[\integ{t}{s}|\nabla X^{-1}_{r,s}(y)|^2\int_{\E}(1\wedge|e|^2)\lambda(de)dr]\\
  &\leq& K(s-t),
\end{array}$$
since $E[|\nabla X^{-1}_{r,s}(y)|^2]<\infty$. Therefore,
$E[|\Gamma^{t,y}_s|^2]\leq 4K(s-t)$. Hence, $$1-2\sqrt{K(s-t)}\leq
E[J(X^{-1}_{t,s}(y))]\leq 1+2\sqrt{K(s-t)}.$$ and the desired result
follows.

\subsection{Regularity of the solution of BSDEs with jumps}
\label{appendix:regularityBSDE} In this section, we are going to prove 
regularity  results for the solution of BSDEs with jumps  with respect to the parameter $(t,x)$ in
order to relate the solution of BSDEs to the classic solution of
PIDEs.  We note that some part of the results given in this section were established in a preprint  of Buckdahn and Pardoux  (1994) \cite{bp}. However, for convenience of the reader and for completeness of the paper, we give the whole proofs.
We first start by giving the $L^p$-estimates  for the solution of the following BSDEs with jumps: 
\begin{equation}\label{bdsde}
Y_t=\xi+\int_t^T f(s,Y_s,Z_s,V_s)ds-\int_t^T
Z_sdW_s-\int_t^T\int_{\E} V_s(e)\widetilde\mu(ds,de).
\end{equation}
\begin{theorem} Assume that $f$ is uniformly Lipschitz with respect to $(y,z,v)$ and
additionally that for some $p\geq 2$, $\xi\in L^{p}_m({\cal F}_T)$ and 
\beq\label{inteoff}E\int_0^T|f(t,0,0,0)|^{p}dt<\infty.\eeq Then 
\begin{equation*} E\bigg[\sup\limits_{t}|Y_t|^{p}+\bigg(\int_0^T|Z_t|^2dt\bigg)^{p/2}+\bigg(\int_0^T  \Big(\int_{\E} |V_t(e)|^2  \lambda (de)\Big) dt\bigg)^{p/2}\bigg]<\infty.\end{equation*}
\end{theorem}
\begin{proof}
We follow the idea of Buckdahn and Pardoux \cite{bp}. The proof is divided into 3 steps.\\ 
\vspace{0.2mm}
Step 1:
From (\ref{bdsde}),
\begin{equation*}
Y_t=Y_0-\int_0^t f(s,Y_s,Z_s,V_s)ds+\int_0^tZ_sdW_s+\int_0^t\int_{\E} V_s(e)\widetilde\mu(ds,de).
\end{equation*}
Then by It\^{o}'s formula,
\beqn
\displaystyle|Y_t|^2&=&|Y_0|^2-2 \int_0^t(f(s,Y_s,Z_s,V_s),Y_s)ds+\int_0^t(|Z_s|^2+\|V_s\|^2)ds\\
&+&2\int_0^t(Y_s,Z_sdW_s)+\int_0^t\int_{\E}(|Y_{s-}+V_s(e)|^2-|Y_{s-}|^2)\widetilde\mu(ds,de).
\eeqn

Let $\phi_{n,p}(x)=(x\wedge n)^p+pn^{p-1}(x-n)^+$ for all $p\geq 1$. Then $\phi_{n,p}\in C^{1}(\mathbb R_+)$, $\phi'_{n,p}(x)=p(x\wedge n)^{p-1}$ bounded and absolutely continuous with
$$\phi''_{n,p}(x)=p(p-1)(x\wedge n)^{p-2}\textbf{1}_{[0,n]}(x).$$

Again by applying It\^{o}'s formula to $\phi_{n,p}(|Y_t|^2)$,
\beqn
\phi_{n,p}(|Y_T|^2)&=&\phi_{n,p}(|Y_t|^2)-2\int_t^T\phi'_{n,p}(|Y_s|^2)(f(s,Y_s,Z_s,V_s),Y_s)ds
+2\int_t^T\phi'_{n,p}(|Y_s|^2)(Y_s,Z_sdW_s)\\&+&\int_t^T\phi''_{n,p}(|Y_s|^2)(Z_sZ_s^\ast Y_s,Y_s)ds
+\int_t^T\phi'_{n,p}(|Y_s|^2)|Z_s|^2ds+\int_t^T\phi'_{n,p}(|Y_s|^2)\|V_s\|^2ds\\
&+&\int_t^T\int_{\E}[\phi_{n,p}(|Y_{s-}+V_s(e)|^2)-\phi_{n,p}(|Y_{s-}|^2)]\widetilde\mu(ds,de)\\
&+&\int_t^T\int_{\E}[\phi_{n,p}(|Y_{s-}+V_s(e)|^2)-\phi_{n,p}(|Y_{s-}|^2)]-(|Y_{s-}+V_s(e)|^2-|Y_{s-}|^2)\phi'_{n,p}(|Y_{s-}|^2)\lambda(de)ds.
\eeqn
Since $(Y,Z,V)\in\mathcal B^2$, it follows from Burkholder-Davis-Gundy's inequality,
$$\displaystyle E\left[\underset{0\leq t\leq T}{\text{sup}}\left|\int_0^t\phi'_{n,p}(|Y_s|^2)(Y_s,Z_sdW_s)\right|\right]\leq Cpn^{p-1}\left\|Y\right\|_{\mathcal S^2}\left\|Z\right\|_{\mathcal H^2},$$
that the $dW$ integral above is uniformly integrable, hence it is a martingale with zero expectation. From the boundedness of $\phi'_{n,p}$, the integrand in the $\widetilde{\mu}$-integral can be written as follows
$$\phi_{n,p}(|Y_{s-}+V_s(e)|^2)-\phi_{n,p}(|Y_{s-}|^2)= \psi_s(e)(2(Y_{s-},V_s(e))+|V_{s}(e)|^2),$$
where $\psi_s(e)$ is a bounded and predictable process. By BDG's inequality
\beqn
\displaystyle E\left[\underset{0\leq t\leq T}{\text{sup}}\left|\int_0^t\int_{\E}\psi_s(e)(Y_{s-},V_s(e))\widetilde{\mu}(ds,de)\right|\right]&\leq& cE\left[\bigg(\int_0^T\int_{\E}|\psi_s(e)(Y_{s-},V_s(e))|^2\mu(ds,de)\bigg)^{1/2}\right]\\
&\leq&cpn^{p-1}\left\|Y\right\|_{\mathcal S^2}\left\|V\right\|_{\mathcal L^2},
\eeqn
and by the decomposition $\widetilde{\mu}(ds,de)=\mu(ds,de)-ds\lambda(de)$,
\beqn
\displaystyle E\left[\underset{0\leq t\leq T}{\text{sup}}\left|\int_0^t\int_{\E}\psi_s(e)|V_s(e)|^2\widetilde{\mu}(ds,de)\right|\right]&\leq& 2pn^{p-1}E\left[\int_0^T\int_{\E}|V_s(e)|^2\lambda(de)ds\right]\\
&=&2pn^{p-1}\left\|V\right\|_{\mathcal L^2}^2.
\eeqn
We have also
$$\displaystyle E\left[\int_0^t\int_{\E}(\phi_{n,p}(|Y_{s-}+V_s(e)|^2)-\phi_{n,p}(|Y_{s-}|^2))\widetilde{\mu}(ds,de)\right]=0.$$
From Taylor's expansion of $\phi_{n,p}$ and the positivity of $\phi''_{n,p}$, we conclude
$$\int_t^T\int_{\E}[\phi_{n,p}(|Y_{s-}+V_s(e)|^2)-\phi_{n,p}(|Y_{s-}|^2)-(|Y_{s-}+V_s(e)|^2-|Y_{s-}|^2)\phi'_{n,p}(|Y_{s-}|^2)]\lambda(de)ds\geq 0.$$
Using again $\phi''_{n,p}\geq 0$ we conclude that:
\beqn
E\phi_{n,p}(|Y_t|^2)&+&E\int_t^T\phi'_{n,p}(|Y_s|^2)(|Z_s|^2+||V_s||^2)ds\\
&\leq&E\phi_{n,p}(|\xi|^2)+2E\int_t^T\phi'_{n,p}(|Y_s|^2)(f(s,Y_s,Z_s,V_s),Y_s)ds\\
&\leq&E\phi_{n,p}(|\xi|^2)+CE\int_t^T\phi'_{n,p}(|Y_s|^2)|Y_s|(|f(s,0,0,0)|+|Y_s|+|Z_s|+\|V_s\|)ds.
\eeqn
Hence we deduce that 
\beqn
E\phi_{n,p}(|Y_t|^2)&+&\frac{1}{2}E\int_t^T\phi'_{n,p}(|Y_s|^2)(|Z_s|^2+\|V_s\|^2)ds\\
&\leq&E\phi_{n,p}(|\xi|^2)+CE\int_t^T\phi'_{n,p}(|Y_s|^2)(|f(s,0,0,0)|^2+|Y_s|^2)ds\\
&\leq&E\phi_{n,p}(|\xi|^2)+CE\int_t^T(|f(s,0,0,0)|^{2p}+\phi'_{n,p}(|Y_s|^2)^{\frac{p}{p-1}}+\phi'_{n,p}(|Y_s|^2)|Y_s|^2)ds\\
&\leq&E\phi_{n,p}(|\xi|^2)+CE\int_t^T(|f(s,0,0,0)|^{2p}+\phi_{n,p}(|Y_s|^2))ds.
\eeqn
Then it follows from Gronwall's lemma that there exists a constant $C(p,T)$ independent of $n$ such that 
$$\underset{0\leq t\leq T}{\text{sup}}E\phi_{n,p}(|Y_t|^2)\leq C(p,T)E[|\xi|^{2p}+\int_0^T|f(t,0,0,0)|^{2p}dt],$$
hence from Fatou's lemma 
\beq\label{controlY}
\underset{0\leq t\leq T}{\text{sup}}E[|Y_t|^{2p}]<\infty,
\eeq
and also
\beq\label{controlZV}
E\int_0^T|Y_s|^{2(p-1)}(|Z_s|^2+\|V_s\|^2)ds<\infty,
\eeq
\beq
\label{controlV}
E\int_0^T\int_{\E}[|Y_{s-}+V_s(e)|^{2p}-|Y_{s-}|^{2p}-p(|Y_{s-}+V_s(e)|^2-|Y_{s-}|^2)|Y_{s-}|^{2(p-1)}]\lambda(de)ds<\infty ,
\eeq
and this holds for any $p \geq 1$.\\
\vspace{0.5mm}
Step 2: Now, again from It\^o's formula,
\beqn
\displaystyle|Y_T|^{2p}&\geq&|Y_t|^{2p}-2p\int_t^T|Y_s|^{2(p-1)}(f(s,Y_s,Z_s,V_s),Y_s)ds\\
&+&p\int_t^T|Y_s|^{2(p-1)}(|Z_s|^2+\|V_s\|^2)ds+2\int_t^T|Y_s|^{2(p-1)}(Y_s,Z_sdW_s)\\
&+&\int_t^T\int_{\E}(|Y_{s-}+V_s(e)|^{2p}-|Y_{s-}|^{2p})\widetilde\mu(ds,de).
\eeqn

It follows from (\ref{controlY}) and (\ref{controlZV}) that the above $dW$-integral is a uniformly integral martingale, and from (\ref{controlZV}) and (\ref{controlV}) that the $\widetilde{\mu}$-integral is a uniformly integrable martingale. It is then easy to conclude that 
$$E[\underset{0\leq t\leq T}{\text{sup}}|Y_t|^{2p}]<\infty,\ p\geq 1.$$
\vspace{0.2mm}
Step 3: Finally, 
\begin{equation*}
\int_s^tZ_rdW_r+\int_s^t\int_{\E} V_r(e)\widetilde\mu(dr,de)=Y_t-Y_s+\int_s^t f(r,Y_r,Z_r,V_r)dr,
\end{equation*}
and from BDG inequality, (\ref{controlY}) and (\ref{inteoff}), for any $p\geq 2$, there exists $C_p$ such that for all $0\leq s\leq t\leq T$, $n\geq 1$ if $\tau_n=\text{inf}\{u\geq s, \int_s^u(|Z_r|^2+\|V_r\|^2)dr\geq n\}\wedge t $,
\beqn
\displaystyle E\left[(\int_s^{\tau_n}(|Z_r|^2+\|V_r\|^2)dr)^{p/2}\right]&\leq&C_p E\left[1+(\int_s^{\tau_n}(|Z_r|+\|V_r\|)dr)^{p}\right]\\
&\leq&C_p (1+(t-s)^{p/2}E\left[(\int_s^{\tau_n}(|Z_r|^2+\|V_r\|^2)dr)^{p/2}\right].
\eeqn
Hence if $C_p (t-s)^{p/2}<1,\ n\geq 1$, $\displaystyle E\left[(\int_s^{\tau_n}(|Z_r|^2+\|V_r\|^2)dr)^{p/2}\right]\leq \frac{C_p}{1-C_p(t-s)^{p/2}}$. It clearly follows that
$$\displaystyle E\left[\bigg(\int_0^T|Z_t|^2dt\bigg)^{p/2}+\bigg(\int_0^T\|V_t\|^2dt\bigg)^{p/2}\bigg)\right]<\infty, \ p\geq 2.$$
\end{proof}
\medskip
From now on, we denote by $\Sigma=(Y,Z,V)$ and ${\mathcal B}^p$ the
space of solutions, {\it i.e.},\beqn ||\Sigma||^p_{{\mathcal
B}^p}\triangleq
E\bigg[\sup\limits_{t}|Y_t|^p+\bigg(\int_0^T|Z_t|^2dt\bigg)^{p/2}+\bigg(\int_0^T||V_t||^2dt\bigg)^{p/2}\bigg].\eeqn
In the sequel, we will consider a specific class of BSDE where
\beqn\xi=g(X_{t,T}(x))
\,\,\mbox{and}\,\,f(s,y,z,v)=f(s,X_{t,s}(x),y,z,v)\eeqn and we assume 
$$\begin{array}{cc}
\,\,\,\,\mathbf{(H)}   \left\{\begin{array}{l}
g\in C_p^3({\bbR^d;\bbR^m}),\\
\forall s\in[0,T], (x,y,z,v)\mapsto f(s,x,y,z,v)\in C^3 \\\mbox{ and
all their derivatives are bounded.}
\end{array}
\right.
\end{array}$$
Note that $f$ is differentiable w.r.t. $v$ in the sens of
Fr\'echet and its Fr\'echet differential is bounded with the norm
in $L^2(\E,\lambda; \mathbb R^k)$.
Let $(Y_s^{t,x},Z_s^{t,x},V_s^{t,x})_{t\leq s\leq T}$ denote the
unique solution of the following BSDE:
\begin{equation}\label{bdsde4}
  \begin{array}{ll}
\displaystyle Y_s^{t,x}=g(X_{t,T}(x))+\int_s^T
f(r,X_{t,r}(x),Y_r^{t,x},Z_r^{t,x},V_r^{t,x})dr-\int_s^T
Z_r^{t,x}dW_r-\int_s^T\int_{\E} V_r^{t,x}(e)\widetilde\mu(dr,de).
  \end{array}
\end{equation}
It follows easily from the existence result in \cite{bbp97}:
\begin{corollary}
 For each $t\in[0,T]$, $x\in\bbR^d$, the BSDE(\ref{bdsde4}) has a
 unique solution $$\Sigma^{t,x}=(Y^{t,x},Z^{t,x},V^{t,x})\in {\cal B}^2,$$ and
 $Y_t^{t,x}$ defines a deterministic mapping from
 $[0,T]\times\bbR^d$ into $\bbR^m$.
\end{corollary}
Now we are going to deal with the regularity of the solution with
respect to the parameter $x$. Let us establish the following proposition:
\begin{proposition}
Under the assumption in the previous theorem and assume moreover
that {\bf{(H)}} holds. Then, for any $p\geq 2$, there exists $C_p$,
$q$ such that for any $0\leq t\leq T$, $x,x'\in \bbR^d$,
$h,h'\in\bbR\setminus\{0\}$, $1\leq i\leq d$,\begin{enumerate}
                                               \item[(i)]$\|\Sigma^{t,x}-\Sigma^{t,x'}\|_{\cB^p}^p\leq
                                               C_p(1+|x|+|x'|)^q|x-x'|^p$;

                                               \item[(ii)]$\|\Delta_h^i\Sigma^{t,x}-\Delta_{h'}^i\Sigma^{t,x'}\|_{\cB^p}^p\leq
C_p(1+|x|+|x'|+|h|+|h'|)^q(|x-x'|^p+|h-h'|^p)$
                                             \end{enumerate}

where
$\Delta_h^i\Sigma_s^{t,x}=\frac{1}{h}(Y_s^{t,x+he_i}-Y_s^{t,x},
Z_s^{t,x+he_i}-Z_s^{t,x},V_s^{t,x+he_i}-V_s^{t,x})$, and
$(e_1,\cdots,e_d)$ is an orthonormal basis of $\bbR^d$.
\end{proposition}
\noindent
\begin{proof} Note that after applying $L^p$-estimation
of the solution to the present situation, we can deduce that
$\forall p\geq 2$, there exist $C_p$, $q$ such that 
\beqn
E\bigg[\sup\limits_{s}|Y_s^{t,x}|^p+\bigg(\int_t^T|Z_s^{t,x}|^2ds\bigg)^{p/2}+\bigg(\int_t^T\|V_s^{t,x}\|^2ds\bigg)^{p/2}\bigg]
\leq C_p(1+|x|^q). \eeqn
For $t\leq s\leq T$, 
\beqn
        Y_s^{t,x}-Y_s^{t,x'}
       &=&g(X_{t,T}(x))-g(X_{t,T}(x'))\\&+&\int_s^T(f(r,X_{t,r}(x),Y_r^{t,x},Z_r^{t,x},V_r^{t,x})-f(r,X_{t,r}(x'),Y_r^{t,x'},Z_r^{t,x'},V_r^{t,x'}))dr\\
        &-&\int_s^T(Z_r^{t,x}-Z_r^{t,x'})dW_r-\int_s^T\int_{\E}(V_r^{t,x}(e)-V_r^{t,x'}(e))\widetilde\mu(dr,de)\\
       &=&\int_0^1g'(\theta X_{t,T}(x)+(1-\theta)X_{t,T}(x'))(X_{t,T}(x)-X_{t,T}(x'))d\theta\\&+&\int_s^T\bigg[\varphi_r(t,x,x')(X_{t,r}(x)-X_{t,r}(x'))
        +\psi_r(t,x,x')(Y_r^{t,x}-Y_r^{t,x'})\\&+&\xi_r(t,x,x')(Z_r^{t,x}-Z_r^{t,x'})+\langle\eta_r(t,x,x'),(V_r^{t,x}-V_r^{t,x'})\rangle\bigg]dr\\
          &-&\int_s^T(Z_r^{t,x}-Z_r^{t,x'})dW_r-\int_s^T\int_{\E}(V_r^{t,x}(e)-V_r^{t,x'}(e))\widetilde\mu(dr,de),
\eeqn where 
$$\begin{array}{cc}
\displaystyle\varphi_r(t,x,x')=\int_0^1\frac{\partial f}{\partial
x}(\Xi_{r,\theta}^{t,x,x'})d\theta,\quad
\displaystyle\psi_r(t,x,x')=\int_0^1\frac{\partial f}{\partial y}(\Xi_{r,\theta}^{t,x,x'})d\theta\\
\displaystyle\xi_r(t,x,x')=\int_0^1\frac{\partial f}{\partial
z}(\Xi_{r,\theta}^{t,x,x'})d\theta,\quad
\displaystyle\eta_r(t,x,x')=\int_0^1\frac{\partial f}{\partial
v}(\Xi_{r,\theta}^{t,x,x'})d\theta,\\
\Xi_{r,\theta}^{t,x,x'}=\theta\Sigma_r^{t,x}+(1-\theta)\Sigma_r^{t,x'},
\end{array}$$
and $\langle\cdot,\cdot\rangle$ stands for the scalar product in
$L^2(\E,\lambda; \bbR^m)$.\smallskip
Then given the boundedness of the derivatives, the previous
proposition and \beqn |\varphi_r(t,x,x')|\leq
C(1+|X_{t,r}(x)|+|X_{t,r}(x')|)^q,\eeqn we know that $(i)$ holds
true after applying the $L^p$-estimates of the solution.\smallskip
Now we turn to $(ii)$. In fact, we have \beqn
        \Delta_h^iY_s^{t,x} &=&\frac{1}{h}(Y_s^{t,x+he_i}-Y_s^{t,x})\\
       &=&\int_0^1g'(\theta X_{t,T}(x+he_i)+(1-\theta)X_{t,T}(x))\Delta_h^iX_{t,T}(x)d\theta\\
        &+&\int_s^T\bigg[\varphi_r(t,x+he_i,x)\Delta_h^iX_{t,r}(x)+\psi_r(t,x+he_i,x)\Delta_h^iY_r^{t,x}\\
        &+&\xi_r(t,x+he_i,x)\Delta_h^iZ_r^{t,x}
        +\langle \eta_r(t,x+he_i,x),\Delta_h^iV_r^{t,x}\rangle\bigg]dr\\
        &-&\int_s^T\Delta_h^iZ_r^{t,x}dW_r-\int_s^T\int_{\E}\Delta_h^iV_r^{t,x}(e)\widetilde\mu(dr,de).
\eeqn Note that for each $p\geq 2$, there exists $C_p$ such that
$E[\sup\limits_{s}|\Delta_h^iX_{t,s}(x)|^p]\leq C_p$. Then same
calculation as in $(i)$ implies that \beqn \|\Delta_h^i
\Sigma^{t,x}\|^p_{{\mathcal B}^p} \leq C_p(1+|x|^q+|h|^q). \eeqn
Finally, we consider \beqn
        \Delta_h^iY_s^{t,x}-\Delta_{h'}^iY_s^{t,x'}
       &=& \Gamma^{i,t,s}(h,x;h',x')\\
        &+&\int_s^T\bigg[\varphi_r(t,x+he_i,x)(\Delta_h^iX_{t,r}(x)-\Delta_{h'}^iX_{t,r}(x'))+
        \psi_r(t,x+he_i,x)(\Delta_h^iY_r^{t,x}-\Delta_{h'}^iY_r^{t,x'})\\
        &+&\xi_r(t,x+he_i,x)(\Delta_h^iZ_r^{t,x}-\Delta_{h'}^iZ_r^{t,x'})
        +\langle \eta_r(t,x+he_i,x),\Delta_h^iV_r^{t,x}-\Delta_{h'}^iV_r^{t,x'}\rangle\bigg]dr\\
         & - &\int_s^T(\Delta_h^iZ_r^{t,x}-\Delta_{h'}^iZ_r^{t,x'})dW_r
        -\int_s^T\int_{\E}(\Delta_h^iV_r^{t,x}(e)-\Delta_{h'}^iV_r^{t,x'}(e))\widetilde\mu(dr,de),\eeqn
where 
\beqn
\Gamma^{i,t,s}(h,x;h',x') &= &\int_0^1\big(g'(\theta
X_{t,T}(x+he_i)+(1-\theta)X_{t,T}(x))\Delta_h^iX_{t,T}(x)\\ &-&g'(\theta X_{t,T}(x'+h'e_i)+(1-\theta)X_{t,T}(x'))\Delta_{h'}^iX_{t,T}(x')\big)d\theta \\
&+&\int_s^T\bigg[(\varphi_r(t,x+he_i,x)-\varphi_r(t,x'+h'e_i,x'))\Delta_{h'}^iX_{t,r}(x')\\
&+&(\psi_r(t,x+he_i,x)-\psi_r(t,x'+h'e_i,x'))\Delta_{h'}^iY_r^{t,x'}\\
&+&(\xi_r(t,x+he_i,x)-\xi_r(t,x'+h'e_i,x'))\Delta_{h'}^iZ_r^{t,x'}\\
&+&\langle\eta_r(t,x+he_i,x)-\eta_r(t,x'+h'e_i,x'),\Delta_{h'}^iV_r^{t,x'}\rangle\bigg]dr.
\eeqn
Since $f$ has bounded derivatives, we have \beqn
&&|\psi_r(t,x+he_i,x)-\psi_r(t,x'+h'e_i,x')|+|\xi_r(t,x+he_i,x)-\xi_r(t,x'+h'e_i,x')|\\
&&+|\eta_r(t,x+he_i,x)-\eta_r(t,x'+h'e_i,x')|
\leq 
C(|\Sigma_r^{t,x+he_i}-\Sigma_r^{t,x'+h'e_i}|+|\Sigma_r^{t,x}-\Sigma_r^{t,x'}|)\eeqn
and \beqn &&|\varphi_r(t,x+he_i,x)-\varphi_r(t,x'+h'e_i,x')|\\
 &\leq&
C\bigg(1+|\Sigma_r^{t,x+he_i}|+|\Sigma_r^{t,x'+h'e_i}|+|\Sigma_r^{t,x}|+|\Sigma_r^{t,x'}|\bigg)^q
\bigg(|\Sigma_r^{t,x+he_i}-\Sigma_r^{t,x'+h'e_i}|+|\Sigma_r^{t,x}-\Sigma_r^{t,x'}|\bigg),\eeqn
it follows from the previous proposition that there exist a constant
$C_p$ and some $\alpha_p$ such that \beqn E[\sup_{t\leq s\leq
T}|\Gamma^{i,t,s}(h,x;h',x')|^p]\leq
C_p(1+|x|+|x'|+|h|+|h'|)^{\alpha_p}(|x-x'|^p+|h-h'|^p).\eeqn
Then, statement (ii) follows from the same estimate, which ends the
proof.
\end{proof}
\medskip
Therefore, using Kolmogorov's criterion, we can get that  $Y$, $Z$,
$V$ are a.s. differentiable w.r.t. $x$. Iterating the same argument,
we get in fact:
\begin{theorem}
For all $0\leq t\leq T$, there exists a version of
$\Sigma^{t,x}=(Y^{t,x},Z^{t,x},V^{t,x})$ such that $x\mapsto
\Sigma^{t,x}$ is a.e. of class $C^2$ from $\bbR^d$ into
$D([t,T];\bbR^m)\times L^2([t,T];\bbR^{m\times d})\times
L^2([t,T]\times \E,ds\lambda(de);\bbR^m)$, where $D([t,T];\bbR^m)$
denotes the set of $\bbR^m$-valued c\`{a}dl\`{a}g functions on
$[t,T]$. Moreover, for any $p\geq 2$, there exist $C_p$ and $q$
such that
$$\begin{array}{ll}
\displaystyle\|\frac{\partial^{|\alpha|}}{\partial
x^{\alpha}}\Sigma^{t,x}\|^p_{\cB^p}\leq C_p(1+|x|^q),\quad 0\leq
|\alpha|\leq 2,\\
\displaystyle\|\frac{\partial^2}{\partial x_i\partial
x_j}\Sigma^{t,x}-\frac{\partial^2}{\partial x_i\partial
x_j}\Sigma^{t,x'}\|^p_{\cB^p}\leq C_p(1+|x|+|x'|)^q|x-x'|^p,\quad
0\leq t\leq T.
\end{array}
$$
\end{theorem}
Additionally, we have the following regularity in $t$:
\begin{proposition}
For all $p\geq 2$ and $s\geq t\vee t'$, there exist $C_p$ and $q$
such that
$$\begin{array}{ll}
\displaystyle\|\Sigma_s^{t,x}-\Sigma_s^{t',x'}\|^p_{\cB^p}\leq
C_p(1+|x|+|x'|)^q(|t-t'|^{p/2}+|x-x'|^p),\\
\displaystyle\|\frac{\partial^{|\alpha|}}{\partial
x^{\alpha}}\Sigma_s^{t,x}-\frac{\partial^{|\alpha|}}{\partial
x^{\alpha}}\Sigma_s^{t',x'}\|^p_{\cB^p}\leq
C_p(1+|x|+|x'|)^q(|t-t'|^{p/2}+|x-x'|^p),\quad 0\leq |\alpha|\leq 2.
\end{array}
$$
\end{proposition}
Hence, we have
\begin{corollary}
The function $(t,x)\mapsto Y_t^{t,x}\in C_p^{1,2}([0,T]\times\bbR^d;
\bbR^m)$.
\end{corollary}
\subsection{Proof of Theorem \ref{main}: existence and uniqueness of solution of PIDEs \eqref{pde1}}
\label{appendix:mainPIDE}
(i) \textit{Uniqueness }: Let $u^1,u^2\in {\mathcal H}_T$ be two
weak solutions of (\ref{pde1}), then Proposition \ref{weak:Itoformula}
implies that for $i=1,2$
$$\begin{array}{ll}
\displaystyle\int_{\R^d}\int_s^Tu^i(r,x)d\phi_t(r,x)dx+(u^i(s,x),\phi_t(s,x))-(g(x),\phi_t(T,x))
-\int_s^T(u^i(r,\cdot),{\mathcal
L}^\ast \phi_t(r,\cdot))dr\\=\displaystyle\int_{\R^d}\int_s^Tf(r,x,u^i(r,x),\sigma^\ast\nabla
u^i(r,x),u^i(r,x+\beta(x,\cdot))-u^i(r,x))\phi_t(r,x)drdx.
\end{array}$$
By the decomposition of the semimartingale $\phi_t(s,x)$, we have
\begin{equation}
\label{Ito:faible}
\begin{split}
&\int_{\R^d}\int_s^Tu^i(r,x)d\phi_t(r,x)dx \\=&
\int_{\R^d}\displaystyle\int_s^T u^i(r,x) {\mathcal
L}^\ast\phi_t(r,x)dr
-\sum_{j=1}^{d}\int_{\R^d}\int_s^T\left(\sum_{i=1}^{d}\frac{\partial}{\partial
x_i}(\sigma^{ij}(x) u^i(r,x) \phi_t(r,x))\right)dW_r^j \\ +&
\int_{\R^d}\int_s^T\int_{\E}  u^i(r,x) {\mathcal
A}^\ast_e\phi_t(r-,x)\dtmur.
\end{split}
\end{equation}

We substitute this in the above equation and get
$$\begin{array}{ll}
\displaystyle\int_{\R^d}u^i(s,x)\phi_t(s,x)dx&=\displaystyle\int_{\R^d}g(x)\phi_t(T,x)dx-\displaystyle\int_s^T\int_{\R^d}(\sigma^\ast\nabla
u^i)(r,x)\phi_t(r,x)dxdW_r\\&-\displaystyle\int_s^T\int_{\E}\int_{\R^d}
u^i(r,x) {\mathcal A}^\ast_e\phi_t(r-,x)dx\dtmur\\
&+\displaystyle\int_s^T\int_{\R^d}f(r,x,u^i(r,x),\sigma^\ast\nabla
u^i(r,x),u^i(r,x+\beta(x,\cdot))-u^i(r,x))\phi_t(r,x)drdx.
\end{array}$$
Then by the change of variable, we obtain
$$\begin{array}{ll}
&\displaystyle\int_{\R^d}u^i(s,X_{t,s}(x))\phi(x)dx\\&=\displaystyle\int_{\R^d}g(X_{t,T}(x))\phi(x)dx
-\displaystyle\int_{\R^d}\int_s^T\phi(x)(\sigma^\ast\nabla
u^i)(r,X_{t,r}(x))dxdW_r\\
&\displaystyle-\int_{\R^d}\int_s^T\int_{\E}\phi(x)(u^i(s,X_{t,r-}(x)+\beta(X_{t,r-}(x),e))-u(s,X_{t,r-}(x)))\dtmur
dx\\
&\displaystyle+\int_{\R^d}\int_s^T\phi(x)f(r,X_{t,r}(x),u^i(r,X_{t,r}(x)),\sigma^\ast\nabla
u^i(r,X_{t,r}(x)),{\mathcal A}_eu^i(r,X_{t,r-}(x)))drdx.
\end{array}$$
Since $\phi$ is arbitrary, we deduce that
$(Y_s^{i,t,x},Z_s^{i,t,x},V_s^{i,t,x})$ solve the BSDE associated
with $(g(X_{t,T}(x)),f)$, $\rho(x)dx$-a.e., where
$$\begin{array}{cc}
Y_s^{i,t,x}=u^i(s,X_{t,s}(x)),Z_s^{i,t,x}=(\sigma^\ast\nabla
  u^i)(s,X_{t,s}(x))\mbox{ and
  }\\V_s^{i,t,x}=u^i(s,X_{t,s-}(x)+\beta(X_{t,s-}(x),\cdot))-u^i(s,X_{t,s-}(x)).\end{array}$$
This means $\rho(x)dx$-a.e., we have
$$Y_s^{i,t,x}=g(X_{t,T}(x))+\integ{s}{T}f(r,X_{t,r}(x),Y_r^{i,t,x},Z_r^{i,t,x},
V_r^{i,t,x})dr-\integ{s}{T}Z_r^{i,t,x}dW_r
       -\integ{s}{T}V_r^{i,t,x}(e)\dtmur.$$
 Then the uniqueness follows from the uniqueness of the BSDE.
\\[0.2cm]
\textit{(ii) Existence} : Let us set $$F(s,x)\triangleq
f(s,x,Y_s^{s,x},Z_s^{s,x},V_s^{s,x}).$$ Plugging into the facts that
$f^0\in \mathbf{L}^2_\rho([0,T]\times\R^d)$ and $f$ is Lipschitz
with respect to $(y,z,v)$, then thanks to Proposition \ref{prop:estimate}
we have $F(s,x)\in \mathbf{L}^2_\rho([0,T]\times\R^d)$. Since $g\in
\mathbf{L}^2_\rho(\R^d)$, we can approximate them by a sequence of
smooth functions with compact support $(g_n,F_n)$ such that
\begin{equation}\label{appr}
 \left\{
  \begin{array}{ll}
     g_n  \rightarrow g \mbox{ in } \mathbf{L}^2_\rho(\R^d),\\
  F_n(s,x)\rightarrow  F(s,x) \mbox{ in } \mathbf{L}^2_\rho([0,T]\times\R^d).
  \end{array}
 \right.
\end{equation}
Denote $(Y_s^{n,t,x},Z_s^{n,t,x},V_s^{n,t,x})\in {\mathcal
S}^2\times {\mathcal H}^2\times {\mathcal L}^2$ the solution of the
BSDE associated with $(\xi_n,F_n)$, where $\xi_n=g_n(X_{t,T}(x))$,
i.e.
\begin{equation}\label{bsde2}
 \begin{array}{ll}
Y_s^{n,t,x}=g_n(X_{t,T}(x))+\integ{s}{T}F_n(r,X_{t,r}(x))dr-\integ{s}{T}Z_r^{n,t,x}dW_r
       -\integ{s}{T} \int_{\E} V_r^{n,t,x}(e)\dtmur.
 \end{array}
\end{equation}
Since $F_n$ and $g_n$ are smooth enough, from the regularity result
of the solution with respect to $(t,x)$ (see the  proof in Appendix), we know that $u_n(t,x):=Y_t^{n,t,x}\in C^{1,2}([0,T]\times
\R^d)$ is the classic solution for the following PIDE:
\begin{equation}\label{pde2} \left\{
\begin{array}{ll}
(\partial_t+{\mathcal L})u(t,x)+F_n(t,x)=0\\u(T,x)=g_n(x).
\end{array}\right.
 \end{equation}
Moreover, we know that $v_n(t,x):=Z^{n,t,x}_t=\sigma^\ast \nabla
u_n(t,x)$. Besides, from the flow property, we can deduce that
$Y_s^{n,t,x}=u_n(s,X_{t,s}(x))$ and $Z_s^{n,t,x}=\sigma^\ast \nabla
u_n(s,X_{t,s}(x))$, as well as the representation of the jump part
$$V_s^{n,t,x}:=w_n(s,X_{t,s}(x)):=u_n(s,X_{t,s-}(x)+\beta(X_{t,s-}(x),\cdot))-u_n(s,X_{t,s-}(x)),\,a.s..$$
Applying the equivalence of norms result and the Proposition
\ref{prop:estimate}, we have
$$\begin{array}{ll}
 &\displaystyle \int_{\R^d}\int_t^T(|u_n(s,x)|^2+|\sigma^\ast \nabla
 u_n(s,x)|^2)ds \rho(x)dx\\
 \leq &\displaystyle c \int_{\R^d}\int_t^T E(|u_n(s,X_{t,s}(x))|^2+|\sigma^\ast \nabla
u_n(s,X_{t,s}(x))|^2)ds \rho(x)dx\\
= &\displaystyle c \int_{\R^d}\int_t^T
E(|Y_s^{n,t,x}|^2+|Z_s^{n,t,x}|^2)ds \rho(x)dx\\
\leq &\displaystyle
c \, \Big(\int_{\R^d}E|g_n(X_{t,T}(x))|^2\rho(x)dx+\int_{\R^d}\int_t^T
E|F_n(s,X_{t,s}(x))|^2ds \rho(x)dx \Big)\\
\leq &\displaystyle
c \, \Big(\int_{\R^d}|g_n(x)|^2\rho(x)dx+\int_{\R^d}\int_t^T
|F_n(s,x)|^2ds \rho(x)dx \, \Big)<+\infty,
\end{array}$$
where $c$ is a constant which changes from line to line.
Therefore, we have shown that $\forall n$, $u_n\in {\mathcal H}_T$
solves the PIDE associated with $(g_n,f_n)$, and for every $\phi\in
C_c^{1,\infty}([0,T]\times \R^d)$,
\begin{equation}
\begin{array}{ll}\label{wspde3}
\displaystyle\int_t^T(u_n(s,\cdot),\partial_s\phi(s,\cdot))ds+(u_n(t,\cdot),\phi(t,\cdot))-(g_n(\cdot),\phi(T,\cdot))
-\int_t^T({\mathcal L} u_n(s,\cdot),\phi(s,\cdot))ds
\\=\displaystyle\int_t^T(F_n(s,\cdot),\phi(s,\cdot))ds.
\end{array}
\end{equation}
Now for $m,n\in N$, applying It\^{o}'s formula to
$|Y_s^{m,t,x}-Y_s^{n,t,x}|^2$, we get for any $s\in [t,T]$,
$$\begin{array}{ll}
 &\displaystyle E|Y_s^{m,t,x}-Y_s^{n,t,x}|^2+E\int_s^T
 |Z_r^{m,t,x}-Z_r^{n,t,x}|^2dr+E\int_s^T\int_{\E}
 |V_r^{m,t,x}(e)-V_r^{n,t,x}(e)|^2\lambda(de)dr\\
 =& E|g_m(X_{t,T}(x))-g_n(X_{t,T}(x))|^2\\&+2\displaystyle E\int_s^T (Y_r^{m,t,x}-Y_r^{n,t,x})
 (F_m(r,X_{t,r}(x))-F_n(r,X_{t,r}(x)))dr.
\end{array}$$
Combining the properties of $f_m$ and $f_n$ with the basic
inequality $2ab\leq \varepsilon a^2+\varepsilon^{-1}b^2$, we have:
$$\begin{array}{ll}
 &\displaystyle E|Y_s^{m,t,x}-Y_s^{n,t,x}|^2+E\int_s^T
 |Z_r^{m,t,x}-Z_r^{n,t,x}|^2dr\\
 \leq &\displaystyle
E \, |g_m(X_{t,T}(x))-g_n(X_{t,T}(x))|^2
+\varepsilon E \, \int_s^T|Y_r^{m,t,x}-Y_r^{n,t,x}|^2dr\\
&\displaystyle+\frac{1}{\varepsilon}E\int_s^T|F_m(r,X_{t,r}(x))-F_n(r,X_{t,r}(x))|^2dr.
\end{array}$$
It follows again from the equivalence of norms that
$$\begin{array}{ll}
 &\displaystyle \int_{\R^d}E|Y_s^{m,t,x}-Y_s^{n,t,x}|^2\rho(x)dx\\
 \leq &\displaystyle
\int_{\R^d}E|g_m(X_{t,T}(x))-g_n(X_{t,T}(x))|^2\rho(x)dx
+\varepsilon\int_{\R^d}E\int_s^T|Y_r^{m,t,x}-Y_r^{n,t,x}|^2dr\rho(x)dx\\
&\displaystyle+\frac{1}{\varepsilon}\int_{\R^d}E\int_s^T|F_m(r,X_{t,r}(x))-F_n(r,X_{t,r}(x))|^2dr\rho(x)dx\\
\leq &\displaystyle \varepsilon
C\int_{\R^d}E\int_s^T|Y_r^{m,t,x}-Y_r^{n,t,x}|^2dr\rho(x)dx+C
\int_{\R^d}|g_m(x)-g_n(x)|^2\rho(x)dx\\
&\displaystyle
+\frac{C}{\varepsilon}\int_{\R^d}\int_s^T|F_m(r,x)-F_n(r,x)|^2dr\rho(x)dx.
\end{array}$$
Choosing $\varepsilon$ appropriately, by Gronwall's inequality and
the convergence of $F_n$ and $g_n$, we get as $m,n\rightarrow
\infty$,
$$\displaystyle \sup\limits_{t\leq s\leq T}\int_{\R^d}E|Y_s^{m,t,x}-Y_s^{n,t,x}|^2\rho(x)dx\rightarrow 0,$$
which implies immediately as $m,n\rightarrow \infty$,
$$\displaystyle \int_{\R^d}E\int_s^T|Y_r^{m,t,x}-Y_r^{n,t,x}|^2dr\rho(x)dx+
\int_{\R^d}E\int_s^T|Z_r^{m,t,x}-Z_r^{n,t,x}|^2dr\rho(x)dx\rightarrow
0.$$ Thanks again to the equivalence of norms (\ref{equi2}), we
get as $m,n\rightarrow \infty$:
$$\begin{array}{ll}
&\displaystyle\int_t^T\int_{\R^d}(|u_m(s,x)-u_n(s,x)|^2+|\sigma^\ast\nabla
u_m(s,x)-\sigma^\ast\nabla u_n(s,x)|^2)\rho(x)dxds\\
\leq &\displaystyle
\frac{1}{C}\int_t^T\int_{\R^d}E(|(u_m-u_n)(s,X_{t,s}(x))|^2+|(\sigma^\ast\nabla
u_m-\sigma^\ast\nabla u_n)(s,X_{t,s}(x))|^2)\rho(x)dxds\\
=&\displaystyle
\frac{1}{C}\int_t^T\int_{\R^d}E(|Y_s^{m,t,x}-Y_s^{n,t,x}|^2+|Z_s^{m,t,x}-Z_s^{n,t,x}|^2)\rho(x)dxds\rightarrow
0,
\end{array}$$
which means that $u_n$ is Cauchy sequence in $\mathcal H_T$. Denote
its limit as $u$, then $u\in \mathcal H_T$. Moreover, since
$X_{t,s}(x)$ has at most countable jumps, using again the
equivalence of norms, we can deduce that
$$\begin{array}{ll}
\displaystyle\int_t^T\int_{\R^d}|w_m(s,x)-w_n(s,x)|^2\rho(x)dxds
\rightarrow 0,
\end{array}$$
henceforth there exists at least a subsequence $\{u_{n_k}\}$, such
that $dt\otimes dP\otimes \rho(x)dx$-a.e.,
\beqn\left\{\begin{array}{lll}
Y^{n_k,t,x}_s&=&u_{n_k}(s,X_{t,s}(x))\rightarrow u(s,X_{t,s}(x))=:Y_s^{t,x},\\
Z^{n_k,t,x}_s&=&\sigma^\ast\nabla u_{n_k}(s,X_{t,s}(x))\rightarrow
\sigma^\ast\nabla u(s,X_{t,s}(x))=: Z_s^{t,x},\\
V^{n_k,t,x}_s&=&w_{n_k}(s,X_{t,s}(x))\rightarrow
u(s,X_{t,s-}(x)+\beta(X_{t,s-}(x),\cdot))-u(s,X_{t,s-}(x)):=
V_s^{t,x}.
\end{array}\right.\eeqn Then we get the desired probabilistic
representations (\ref{representation}). Passing limit in
(\ref{wspde3}), we have for every $\phi$,
\begin{equation}
\begin{array}{ll}\label{wspde4}
\displaystyle\int_t^T(u(s,\cdot),\partial_s\phi(s,\cdot))ds+(u(t,\cdot),\phi(t,\cdot))-(g(\cdot),\phi(T,\cdot))
-\int_t^T( u(s,\cdot),{\mathcal L}^\ast \phi(s,\cdot))ds
\\=\displaystyle\int_t^T(f(s,\cdot,u(s,\cdot),\sigma^\ast\nabla
u(s,\cdot),u(s,x+\beta(x,\cdot))-u(s,x)),\phi(s,\cdot))ds,
\end{array}
\end{equation}
which means that $u\in \mathcal H_T$ is the weak solution of
(\ref{pde1}). $\hfill \Box$

\subsection{Proof of the tightness of the sequence $(\pi_n)_{ n\in \mathbb N}$}
\label{appendix:tight}
Recall first that  $\nu _{n}(dt,dx)=n(u_{n}-h)^{-}(t,x)dtdx$ and $\pi _{n}(dt,dx)=\rho
(x)\nu _{n}(dt,dx)$ where $ u_n $ is the solution of the PIDEs \eqref{o-equa1}.  

\begin{lemma}
\label{tight}
 The sequence of measure $(\pi _{n})_{n  \in \mathbb N}$  is tight.
\end{lemma}
\begin{proof}
Since here we need to deal with the additional jump part, we adapt the proof of Theorem 4 in \cite{BCEF}.  
We shall prove that for every $\epsilon>0$ , there exists some constant $K$ such that
\beq
\int_0^T\int_{\mathbb R^d} \textbf{1}_{\left\lbrace \lvert x\rvert\geq 2K\right\rbrace} \pi_n(ds,dx)\leq \epsilon,\ \forall n\in N.
\eeq
We first write
\begin{align*}
&\int_0^T\int_{\mathbb R^d} \textbf{1}_{\left\lbrace \lvert x\rvert\geq 2K\right\rbrace} \pi_n(ds,dx)\\&= \int_0^T\int_{\mathbb R^d} \textbf{1}_{\left\lbrace \lvert x\rvert\geq 2K\right\rbrace} \left( \textbf{1}_{\left\lbrace \left|X^{-1}_{0,s}(x)-x\right|\leq K\right\rbrace} + \textbf{1}_{\left\lbrace \left|X^{-1}_{0,s}(x)-x\right|\geq K\right\rbrace} \right) \pi_n(ds,dx)\\
&:=I^n_K+L^n_K, \quad P-a.s.
\end{align*}
Taking
expectation yields \beqn \int_0^T\int_{\mathbb R^d}
\textbf{1}_{\left\lbrace \lvert x\rvert\geq 2K\right\rbrace}
\pi_n(ds,dx)=E I^n_K+E L^n_K. \eeqn
By (\ref{est-measure}) and  for $K \geq 2\lVert b \rVert_{\infty}T $, we get 
\beqn
E L^n_K &\leq& \int_0^T\int_{\mathbb R^d} \P\left(\underset{0\leq r\leq T}{{\rm sup}}\left|X^{-1}_{0,r}(x)-x\right|\geq K\right) \pi_n(ds,dx)\\
&\leq& \left(  C_1\ {\rm exp}(-C_2K^2)+C_3\ {\rm exp}(-C_4K) \right) \pi_n\left([0,T]\times \mathbb R^d \right)\\
&\leq& C'_1\ {\rm exp}(-C_2K^2)+C'_3\ {\rm exp}(-C_4K), 
\eeqn  so $E L^n_K \leq\epsilon$ for $K$ sufficiently large.
On the other hand, if $\left|x\right|\geq 2K $ and $\left|X^{-1}_{0,s}(x)-x\right|\leq K $ then $\left|X^{-1}_{0,s}(x)\right|\geq K $. Therefore
\beqn
E I^n_K&\leq&E \int_0^T\int_{\mathbb R^d}\textbf{1}_{\left\{\left|X^{-1}_{0,s}(x)\right|\geq K \right\}}\rho(x)\nu_n(ds,dx)\\
&=&E \int_0^T\int_{\mathbb
R^d}\textbf{1}_{\left\{\left|X^{-1}_{0,s}(x)\right|\geq K
\right\}}\rho(x)n(u_{n}-h)^{-}(s,x)dsdx \eeqn which, by the change
of variable $y= X^{-1}_{0,s}(x)$, becomes

\begin{align*}
& E \int_0^T\int_{\mathbb R^d}\textbf{1}_{\left\{\left| y\right|\geq K \right\}}\rho(X_{0,s}(y))J(X_{0,s}(y))n(u_{n}-h)^{-}(s,X_{0,s}(y))dsdy \\ 
&\leq E \int_{\mathbb R^d}\rho(x)\left(\rho(x)^{-1}\textbf{1}_{\left\{\left|x\right|\geq K\right\}} \underset{0\leq r\leq T}{{\rm sup}}\rho(X_{0,r}(x))J(X_{0,r}(x))\right)K^{n,0,x}_{T}dx \\
&\leq \left(E \int_{\mathbb R^d} \left(K^{n,0,x}_{T}\right)^2\rho(x)dx\right)^{1/2}\\ &\hspace{0.5cm}\left(E \int_{\mathbb R^d} \left(\rho(x)^{-1} \textbf{1}_{\left\{\left|x\right|\geq K\right\}}\underset{0\leq r\leq T}{{\rm sup}}\rho(X_{0,r}(x))J(X_{0,r}(x))\right)^2 \rho(x)dx\right)^{1/2} \\
&\leq C \left(E \int_{\mathbb R^d} \left(\rho(x)^{-1}
\textbf{1}_{\left\{\left|x\right|\geq K\right\}}\underset{0\leq
r\leq T}{{\rm sup}}\rho(X_{0,r}(x))J(X_{0,r}(x))\right)^2
\rho(x)dx\right)^{1/2}.
\end{align*}
where the last inequality is a consequence of (\ref{K-estimate}). It is now sufficient to prove that \beq \int_{\mathbb
R^d}\rho(x)^{-1}E\left[\left(\underset{0\leq r\leq T}{{\rm
sup}}\rho(X_{0,r}(x))J(X_{0,r}(x))\right)^2\right]dx<\infty . \eeq

Note that
\begin{align*}
&E\left[\left(\underset{0\leq r\leq T}{{\rm sup}}\rho(X_{0,r}(x))J(X_{0,r}(x))\right)^2\right]\\ &\leq \left[E\left(\underset{0\leq r\leq T}{{\rm sup}}\left|\rho(X_{0,r}(x))\right|\right)^4\right]^{1/2}\left[E\left(\underset{0\leq r\leq T}{{\rm sup}}\left|J(X_{0,r}(x))\right|\right)^4\right]^{1/2}\\
&\leq C\left[E\left(\underset{0\leq r\leq T}{{\rm
sup}}\left|\rho(X_{0,r}(x))\right|\right)^4\right]^{1/2}.
\end{align*}
Therefore it is sufficient to prove that:
\begin{equation*}
\int_{\mathbb R^d}\frac{1}{\rho(x)}\left(E\left[\underset{t\leq
r\leq T}{{\rm
sup}}\left|\rho(X_{t,r}(x))\right|^4\right]\right)^{1/2}dx<\infty.
\end{equation*}
Since $\rho(x)\leq 1$, we have
\beqn
E\left[\underset{t\leq r\leq T}{{\rm sup}}\left|\rho(X_{t,r}(x))\right|^4\right]&\leq& E\left[\underset{t\leq r\leq T}{{\rm sup}}\left|\rho(X_{t,r}(x))\right|^4\textbf{1}_{\left\lbrace \underset{t\leq r\leq T}{{\rm sup}}\left|X_{t,r}(x)-x\right|\leq\frac{\left|x\right|}{2}\right\rbrace }\right]\\
& &+\P\left(\underset{t\leq r\leq T}{{\rm sup}}\left|X_{t,r}(x)-x\right|\geq\frac{\left|x\right|}{2}\right)\\
&=&:A(x)+B(x).
\eeqn
If $\underset{t\leq r\leq T}{{\rm sup}}\left|X_{t,r}(x)-x\right|\leq\frac{\left|x\right|}{2}$ then $\left|X_{t,r}(x)\right|\geq\frac{\left|x\right|}{2}$ and so $\left|\rho(X_{t,r}(x))\right|\leq \left(1+\frac{\lvert x \rvert}{2} \right)^{-p}$. Thus we have that $A(x)\leq \left(1+\frac{\lvert x \rvert}{2} \right)^{-4p}$ and so $\int_{\mathbb R^d} \left(1+\lvert x \rvert \right)^{p}A(x)^{1/2}dx<\infty$. On the other hand, if $\lvert x \rvert \geq 4\lVert b \rVert_{\infty}T$, then (the same argument as in the existence proof step 2 of Theorem 4 in \cite{BCEF} for the It\^o integral with respect to the Brownian motion; and see e.g. Theorem 5.2.9 in \cite{app} for the integral with respect to the compensated Poisson random measure) 
\begin{align*}
\label{large_deviation}
B(x) &\leq \P\left(\underset{t\leq s\leq T}{{\rm sup}}\left|\int_0^s\sigma(X_{0,r}(x))dW_r\right|\geq \frac{\lvert x \rvert}{8} \right)\\
&\hspace{0.3cm} + \P\left(\underset{t\leq s\leq T}{{\rm sup}}\left|\int_0^s\int_{\mathbb E}\beta(X_{0,r-}(x),e)\tilde \mu(dr,de)\right|\geq \frac{\lvert x \rvert}{8} \right)\\
&\leq C_1\ {\rm exp}(-C_2\lvert x \rvert^2)+C_3\ {\rm exp}(-C_4\lvert x \rvert)
\end{align*}
and so $\int_{\mathbb R^d} \left(1+\lvert x \rvert \right)^{p}B(x)^{1/2}dx<\infty$.

\end{proof}

\small{

}


\begin{thebibliography}{99}

\bibitem{app} Applebaum D.: L{\'e}vy Procrsses and Stochastic Calculus (Second Edition). \textit{Cambridge University Press (2009).}

\bibitem{bm} Bally, V., Matoussi, A.: Weak Solutions for SPDEs and Backward Doubly
Stochastic Differential Equations. \textit{Journal of Theoretical
Probability}, Vol. \textbf{14}, No. 1 (2001).

\bibitem{BCEF} Bally, V., Caballero, M.E., El Karoui, N. and Fernandez, B. : Reflected BSDE's, PDE's and Variational Inequalities. \textit{Rapport de Recherche INRIA} No. 4455 (2002).

\bibitem{bbp97} Barles,  G., Buckdahn,  R., Pardoux, E.: Backward stochastic differential equations and
integral-partial differential equations, \textit{Stochastics and
stochastic reports}, Vol.\textbf{60}, pp. 57-83 (1997).

\bibitem{bl} Barles G., Lesigne L.: SDE, BSDE and PDE. In: N.El
Karoui and L.Mazliak (eds.), Backward stochastic differential
equations. \textit{Pitman Research Notes in Mathematics Series} \textbf{364},
pp. 47-80 (1997).

\bibitem{Becherer}  Becherer, D.: Bounded Solutions to Backward
SDE's with jumps for utility optimization and indifference hedging.
 \textit{Annals of Applied Probability}, Vol. \textbf{16}, p. 2027-2054 (2006).

\bibitem{BensoussanLions78} Bensoussan,  A., Lions,  J.-L. Applications des In\'equations variationnelles en contrôle
stochastique. \textit{Dunod}, Paris (1978).

\bibitem{b} Bismut, J.M.: M\'{e}canique al\'{e}atoire, \textit{Ecole d'\'{e}t\'{e} de
Probabilit\'{e} de Saint-Flour,  Lect. Notes Math.} \textbf{929}, 5-100 (1980).




\bibitem{bp} Buckdahn R., Pardoux E.: BSDEs with jumps and associated
integro-partial differential equations. \textit{Preprint} (1994).

\bibitem{C1}  Cr\'{e}pey, S.: About the Pricing Equation in Finance. Forthcoming in \textit{Paris-Princeton Lectures in Mathematical Finance}, Lecture Notes in Mathematics, Springer (2010).



\bibitem{cm} Cr\'{e}pey S., Matoussi A.: Reflected and Doubly Reflected BSDEs with jumps.  \textit{Annals of  Applied Probability} \textbf{18}, No. 5, p. 2041-2069 (2008).


\bibitem{DMZ12} Denis L., Matoussi A. and Zhang J.: The Obstacle Problem for Quasilinear Stochastic PDEs: Analytical approach. \textit{arXiv:1202.3296v1 (2012)},  to appear in \textit{The Annals of Probability}. 
\bibitem{Elk2}  El Karoui, N.,  Kapoudjian, C., Pardoux, E., Peng, S.,   Quenez M.C.:  Reflected Solutions of Backward SDE and Related Obstacle Problems for PDEs.  \textit{Annals of  Probability}
\textbf{25},  702-737 (1997). \MR{1434123}

\bibitem{ElkPaQ97} El Karoui, N., Pardoux, E., Quenez, M.C.: Reflected backward SDEs and American options, {\sl Numerical methods in finance}, Cambridge University press, 215--231 (1997).

\bibitem{ElkMH08}
El Karoui, N., Hamad\`ene, S., Matoussi, A.: Backward stochastic differential equations and applications, {\sl Chapter 8 in the book "Indifference Pricing: Theory and Applications", Springer-Verlag}, 267--320 (2008).

\bibitem{ElKM13}  El Karoui, N., Mrad,M. : Stochastic Utilities With a Given Optimal Portfolio : Approach by Stochastic Flows . \textit{	arXiv:1004.5192} (2013).

\bibitem{ElKM12}  El Karoui, N., Mrad,M. : An Exact Connection between two Solvable SDEs and a Nonlinear Utility Stochastic PDE. \textit{	 arXiv:1004.5191} (2012).

\bibitem{ElPQ97} El Karoui, N., Peng, S., Quenez, M.C : Backward Stochastic Differential Equations in Finance, \textit{Mathematical Finance} Volume 7, Issue 1, pages 1-71, (1997).
 
\bibitem{fk} Fujiwara T., Kunita H.: Stochasitc differential equations
of jump type and L\'{e}vy processes in diffeomorphisms group, \textit{J. Math. Kyoto Univ.} \textbf{25}, pp. 71-106 (1985).

\bibitem{hh} Hamad\`ene, S., Hassani, M.: BSDEs with two reflecting barriers driven by a Brownian and a Poisson noise and related Dynkin game. \textit{ Electronic Journal of Probability} \textbf{11}, No.\textbf{5}, pp. 121-145 (2006).


\bibitem{ho} Hamad\`ene, S.,  Ouknine, Y.: Reflected backward stochastic differential equation with jumps and radom obstacle. \textit{Electronic Journal of Probability}
\textbf{8}, No.\textbf{2}, pp. 1-20 (2003).

\bibitem{iw} Ikeda,  N., Watanabe,  S.: Stochastic differential equations
and diffusion processes, \textit{North-Holland/Kodansha, second
edition.} (1989).

\bibitem{k5} Ishikawa Y., Kunita, H.: Malliavin calculus on the Wiener-Poisson space and its application to canonical SDE with jumps.  \textit{Stochastic Process. Appl.} \textbf{ 116 },  no. 12, p. 1743-1769 (2006).

\bibitem{jacod} Jacod J., Shiryaev A.N.: Limit theorems for stochastic
processes. \textit{Springer, Berlin Heidelberg New York} (2003).

\bibitem{jourdain07} Jourdain B.: Stochastic flows approach to
Dupire's formula, \textit{Finance and Stochastics}, Vol. 11(4), pp
521-535 (2007).

\bibitem{KS80} Kinderlehrer, D.,  Stampacchia, G.:  An Introduction to Variational Inequalities and Their Applications. \textit{Academic Press, Harcourt Brace Jovanivich, Publishers }(1980). 

\bibitem{Klimsiak} Klimsiak T.: Reflected BSDEs and
obstacle problem for semilinear PDEs in divergence form. \textit{Stochastic
Processes and their Applications},  {\bf 122} (1), 134-169
(2012).


\bibitem{k1} Kunita H.: Convergence of stochastic flows with jumps
and L\'{e}vy processes in diffeomorphisms group, \textit{Annales de
l'I.H.P.}, section B, tome \textbf{22}, no.3, p.287-321 (1986).

\bibitem{k2} Kunita H.: Stochastic differential equations and stochastic
flows of diffeomorphisms, \textit{Ecole d'\'{e}t\'{e} de
Probabilit\'{e} de Saint-Flour,  Lect. Notes Math.}, \textbf{1097}, p. 144-303 (1982).

\bibitem{k3} Kunita H.: Stochastic flow acting on Schwartz
distributions, \textit{Journal of Theoretical Probability}, Vol. \textbf{7}, 2, p. 279-308, 1994.

\bibitem{k4} Kunita H.: Stochastic differential equations based on L\'evy processes and stochastic flows of diffeomorphisms. \textit{Real and stochastic analysis,  p. 305-373, Trends Math., Birkh\"auser Boston, MA} (2004).

\bibitem{l} L\'{e}andre R. : Flot d'une \'{e}quation
diff\'{e}rentielle stochastique avec semimartingale directrice
discontinue. \textit{S\'{e}minaire de probabilit\'{e}s (Strasbourg)},
tome \textbf{19}, p.271-274 (1985).

\bibitem{MPS04} Matache A.-M.,  Von Petersdorff T., Schwab C.:  Fast Deterministic Pricing of Options on L\'evy Driven Assets. \textit{ESAIM:  Mathematical Modelling and Numerical Analysis}, Vol. \textbf{38}, No 1, pp. 37-71 (2004).

\bibitem{MS03}  Matoussi, A., Scheutzow, M.  : {Semilinear
Stochastic PDE's with nonlinear noise and Backward Doubly SDE's}.
\textit{Journal of Theoretical
Probability} \textbf{15},  1-39 (2002). \MR{1883201}

\bibitem{MX08} Matoussi,  A.,  Xu, M.:  Sobolev solution for semilinear PDE with obstacle
under monotocity condition. \textit {Electronic Journal of Probability}
\textbf{13}, No.\textbf{35}, pp. 1035-1067 (2008).

\bibitem{MW09} Matoussi,  A.,  Wang, H.:  Probabilistic interpretation for Sobolev solution of semilinear parabolic partial integro-differential 
equations. \textit{Preprint of  University of Maine (Le Mans)} (2009).


\bibitem{MS10} Matoussi,  A., Sto\"{\i}ca,  L.: The Obstacle Problem for Quasilinear Stochastic PDE's.
 \textit{The Annals of Probability}, {\bf{38}}, 3, 1143-1179 (2010).


\bibitem{M81}Meyer, P.-A.: Flot d'une \'equation diff\'erentielle stochastique. \textit{S\'eminaire de probabiliti\'es (Strasbourg)}, tome \textbf{XV}, p. 103-117 (1981).

\bibitem{MignotPuel75} Mignot, F  Puel, J.-P.  :  Solutions maximum de certaines inéquations d'évolution paraboliques et inéquations quasi-variationnelles paraboliques. \textit{C.R.A.S.} 280 série A, page 259 and ARMA (1976).

\bibitem{Mrad09} Mrad M.: Utilit\'es progressives. Th\`ese de Doctorat de l'\'Ecole Polytechniques (2009).

\bibitem{ouknineturpin06} Ouknine, Y. Turpin, I: Weak solutions of semilinear PDEs in Sobolev spaces and their probabilistic interpretation via
the FBSDE. \textit {Stochastic Analysis and Applications}, \textbf{24}: 871-888 (2006).

%
\bibitem{pp1} Pardoux E., Peng S.: Backward stochastic
differential equations and quasilinear parabolic partial equations,
\textit{Lect. Notes Control Inf. Sci.}, vol. \textbf{176}, pp. 200-217 (1992).

\bibitem{pp1994} Pardoux E., Peng S.: Backward doubly SDE's and
systems of quasilinear SPDEs. \textit{Probab. Theory and Related
Field}, \textbf{98}, 209-227 (1994).

\bibitem{PPR} Pardoux E.,  Pradeilles F.,  Rao Z.: Probabilistic interpretation of systems of semilinear PDEs. {\em
Annales de l'Institut Henri Poincar\'e, s\'erie Probabilit\'es-Statistiques},
\textbf{33}, p. 467-490 (1997).

\bibitem{p1} Peng S.: Probabilistic interpretation for systems of quaslinear parabilic PDEs. \textit{Stochastics and Stochastics Reports}, Vol. \textbf{37}, p. 61-74 (1991).

\bibitem{Pierre} Pierre M.: Probl\`emes d'Evolution avec Contraintes Unilaterales et
Potentiels Parabolique. \textit{Comm. in Partial Differential
Equations}, {\bf 4(10)}, 1149-1197 (1979).

\bibitem{PIERRE} Pierre M. : Repr\'esentant Pr\'ecis d'Un Potentiel Parabolique. \textit{S\'eminaire
de Th\'eorie du Potentiel}, Paris, {\bf No.5}, Lecture Notes in
Math. 814, 186-228 (1980).

\bibitem{protter} Protter Ph.: Stochastic integration and differential
equations, \textit{2-nd Edition, Version 2.1, Springer B.H.N-Y} (2000).

\bibitem{royer} Royer M.: BSDEs with jumps and related non linear expectations.
\textit{Stochastics Processes and their Applications}, vol. \textbf{116}, pp. 1358-1376 (2006).

\bibitem{situ} Situ R.: On solutions of backward stochastic
differential equations with jumps and applications,
\textit{Stochastic Processes and their Applications}, \textbf{66},
209-236 (1997).

\bibitem{tangli} Tang S., Li X.: Necessary condition for optimal
control of stochastic systems with random jumps, \textit{SIAM JCO}
\textbf{332}, pp. 1447-1475 (1994).


\end{thebibliography}
\end{document}